\documentclass[12pt]{article}
\usepackage[centertags]{amsmath}
\usepackage{amsfonts}
\usepackage{amssymb}
\usepackage{latexsym}
\usepackage{amsthm}
\usepackage{newlfont}
\usepackage{graphicx}
\usepackage{listings}
\usepackage{booktabs}
\usepackage{abstract}
\usepackage{enumerate}
\usepackage{xcolor}
\RequirePackage{srcltx}
\date{}

\newlength{\defbaselineskip}
\newcommand{\setlinespacing}[1]%
           {\setlength{\baselineskip}{#1 \defbaselineskip}}

\newcommand{\N}{{\mathbb{N}}}

\newcommand{\actaqed}{\hfill $\actabox$}
{\medskip\noindent \textit{Proof of #1. }}%
{\actaqed \medskip}

\def\D{{\mathcal D}}

\def\cA{{\mathcal A}}
\def\C{{\mathcal C}}
\def\cC{{\mathcal C}}
\def\cF{{\mathcal F}}
\def\FF{{\mathcal F}}

\def\cH{{\mathcal H}}
\def\CW{\mathcal{W}}

\def \Tr{\mathcal T}

\def \cN{\mathcal N}

\def \cX{\mathcal X}

\def\R{{\mathbb R}}
\def\Z{\mathbb Z}
\def\ZZ{\mathbb Z}

\def \T{\mathbb T}

\def\bbC{\mathbb C}
\def \<{\langle}
\def\>{\rangle}

\def \Og{\Omega}

\def \e{\varepsilon}
\def \va{\varepsilon}

\def \ff{\varphi}

\def\ga{\gamma}

\def \sp{\operatorname{span}}

\def\bx{\mathbf x}

\def\bz{\mathbf z}
\def\bk{\mathbf k}

\def\bu{\mathbf u}

\def\bs{\mathbf s}

\def\ba{\mathbf a}

\def\bW{\mathbf W}

\def\bF{\mathbf F}
\def\bU{\mathbf U}

\newtheorem{Theorem}{Theorem}[section]
\newtheorem{Lemma}{Lemma}[section]
\newtheorem{Definition}{Definition}[section]

\newtheorem{Remark}{Remark}[section]

\newtheorem{Corollary}{Corollary}[section]
\numberwithin{equation}{section}

\newcommand{\be}{\begin{equation}}
\newcommand{\ee}{\end{equation}}

\def\da{{\delta}}
\def\ga{{\gamma}}
\def\og{{\omega}}
\def\ld{{\lambda}}
\def\Bl{\Bigl}
\def\Br{\Bigr}
\def\f{\frac}

\def\vi{\varphi}

\def\va{\varepsilon}

\def\CD{{\mathcal D}}

\def\CS{{\mathcal S}}

\def\CW{{\mathcal W}}
\def\CC{{\mathbb C}}

\def\NN{{\mathbb N}}
\def\PP{{\mathbf P}}

\def\RR{{\mathbb R}}
\def\ZZ{{\mathbb Z}}

\def\Og{\Omega}

\def\sub{\substack}

\def\Ld{\Lambda}

\def\cN{\mathcal{N}}
\def\cA{\mathcal{A}}

\def\cX{\mathcal{X}}

\def\cH{\mathcal{H}}

\def\FF{\mathcal{F}}

\def\EE{\mathbf{E}}

\def\bu{\mathbf{u}}

\def\spn{\operatorname{span}}
\def\bx{\mathbf{x}}

\def\bz{\mathbf{z}}

\def\bx{\mathbf{x}}

\DeclareFontEncoding{FMS}{}{}
\DeclareFontSubstitution{FMS}{futm}{m}{n}
\DeclareFontEncoding{FMX}{}{}
\DeclareFontSubstitution{FMX}{futm}{m}{n}
\DeclareSymbolFont{fouriersymbols}{FMS}{futm}{m}{n}
\DeclareSymbolFont{fourierlargesymbols}{FMX}{futm}{m}{n}
\DeclareMathDelimiter{\VT}{\mathord}{fouriersymbols}{152}{fourierlargesymbols}{147}

\usepackage{mathtools}

\DeclarePairedDelimiter\floor{\lfloor}{\rfloor}

\begin{document}

\title{Random points are good for universal discretization}

\author{ F. Dai and   V. Temlyakov 	\footnote{
		The first named author's research was partially supported by NSERC of Canada Discovery Grant
		RGPIN-2020-03909.
		The second named author's research was supported by the Russian Science Foundation (project No. 23-71-30001)
at the Lomonosov Moscow State University.
  }}

\newcommand{\Addresses}{{
  \bigskip
  \footnotesize

  F.~Dai, \textsc{ Department of Mathematical and Statistical Sciences\\
University of Alberta\\ Edmonton, Alberta T6G 2G1, Canada\\
E-mail:} \texttt{fdai@ualberta.ca }

 \medskip
  V.N. Temlyakov, \textsc{University of South Carolina,\\ Steklov Institute of Mathematics,\\  Lomonosov Moscow State University,\\ and Moscow Center for Fundamental and Applied Mathematics.
  \\
E-mail:} \texttt{temlyak@math.sc.edu}

}}
\maketitle

\begin{abstract}{

There has been significant progress in the study of sampling discretization of integral norms for  both  a designated  finite-dimensional function space and   a finite collection of such function spaces (universal discretization).  Sampling discretization results  turn out to be very useful  in various applications, particularly in  sampling recovery.
 Recent sampling discretization results  typically provide existence of good sampling points for discretization. 
  In this paper, we show  that  independent and  identically distributed  random  points provide good  universal discretization with high probability. Furthermore, we demonstrate that a simple greedy algorithm based on those  points that are good for universal discretization  provides excellent sparse recovery results in the square norm.

	 }\end{abstract}


{\it MSC classification 2000:} Primary 65J05; Secondary 42A05, 65D30, 41A63.

\section{Introduction}
\label{I}

In this paper, we continue the discussion from \cite{DT} on  universal discretization of integral norms for collections of finite-dimensional subspaces that are  spanned by functions  from  a finite  dictionary of  uniformly bounded functions on a domain.   Let us first  describe  some necessary  notations and  concepts related to  universal discretization for the rest of this paper.

Let  $\Omega$ be a locally compact Hausdorff space   equipped  with a Borel  probability measure $\mu$.  For  $1\le p\leq  \infty$, we denote by  $L_p(\Omega):=L_p(\Omega,\mu)$    the  Lebesgue  space $L_p$
defined with respect to the measure $\mu$ on $\Omega$,  and by  $\|\cdot\|_p$  the  norm of $L_p(\Og)$.   Throughout the paper, we will use a slight abuse of the notation that $L_\infty(\Og)$ denotes the space of all uniformly bounded measurable functions on $\Og$ with norm
$\|f\|_\infty=\sup_{x\in \Og} |f(x)|$.
Discretization of the $L_p$ norm refers to the process of replacing the measure $\mu$ with a discrete measure $\mu_m$ supported  on a finite subset $\xi = \{\xi^j\}_{j=1}^m \subset \Omega$. This involves substituting integration with respect to the measure $\mu$ with an appropriate  sum of evaluations of $f$ 
 at the points $\xi^j$, $j=1,\cdots,m$.  This approach to discretization is commonly referred to as {\it sampling discretization}.
  We now  formulate explicitly   both the sampling discretization problem (also known as the Marcinkiewicz discretization problem) and the problem of universal discretization.

{\bf The sampling discretization problem.} Let 
$X_N\subset L_p$ be an $N$-dimensional subspace of $L_p(\Omega,\mu)$ with  $1\le p< \infty$  (the index $N$ here, usually, stands for the dimension of $X_N$).  We shall always assume that every function in $X_N$ is defined everywhere on $\Og$, and  for every $x\in \Og$, 
\[ \sup\{|f(x)|:\  \ f\in X_N,\  \|f\|_p\leq 1\} <\infty.\]  We say that $X_N$  admits the {\it  Marcinkiewicz-type discretization theorem} with parameters $m\in \N$ and $p$ and positive constants $C_1\le C_2$ if there exists a set $\xi := \{\xi^j\}_{j=1}^m $ of $m$ points in $\Og$
such that for any $f\in X_N$  
\be\label{A1}
C_1\|f\|_p^p \le \frac{1}{m} \sum_{j=1}^m |f(\xi^j)|^p \le C_2\|f\|_p^p.
\ee
This  definition can be slightly modified to include the case $p=\infty$ as well.

{\bf The  problem of universal discretization.} Let $\cX:= \{X(n)\}_{n=1}^k$ be a finite  collection of finite-dimensional  function  spaces on $\Og$.  Given $1\leq p<\infty$, we say that a set $\xi:= \{\xi^j\}_{j=1}^m \subset \Omega $ provides {\it universal discretization} of  the $L_p$ norm for  the collection $\cX$ with  positive constants $C_1, C_2$ if for each $n\in\{1,\dots,k\}$ and any $f\in X(n)$ 
\[
C_1\|f\|_p^p \le \frac{1}{m} \sum_{j=1}^m |f(\xi^j)|^p \le C_2\|f\|_p^p.
\]
This  definition can be slightly modified to include the case $p=\infty$ as well.
 For the sake of convenience in later applications, we  say  a finite  set of points in $ \Omega $   provides {\it   universal discretization}   for  a  collection  $\cX$ of finite-dimensional spaces  if it provides {\it universal discretization} of the $L_2$ norm for the collection $\cX$  with constants $\f12$ and $\f32$; that is, 	\be\label{I3}
 \frac{1}{2}\|f\|_2^2 \le \frac{1}{m} \sum_{j=1}^m |f(\xi^j)|^2\le \frac{3}{2}\|f\|_2^2\quad \text{for any}\quad f\in \bigcup_{V\in\cX} V .
 \ee
 Additionally, 
we  define  $m(\cX)$ to be the smallest positive integer  $m$ for which there exists a set $\xi$ of $m$ points in $\Og$ which
provides  universal discretization (\ref{I3}) for the collection $\cX$. \footnote{We define $m(\cX)=\infty$ if there does not exist such  a finite set  $\xi$ for any positive integer $m$.}

%
%
%
%

 We  point out that the concept of universality is well known in approximation theory. For instance, the reader can find a discussion of universal cubature formulas in \cite{VTbookMA}, Section 6.8.

Given  $x\in \RR$, we use the notation $ \floor*{x}$ (resp.  $\left \lceil{x}\right \rceil $)  to denote  the largest (resp. smallest)  integer $\leq  x$ (resp.,  $\ge x$).

In this paper, our discussion  will focus on  universal discretization
for  special collections of subspaces generated by a finite dictionary of  bounded  functions on $\Og$.
Throughout the paper, we will write our dictionary in the form 
 $\D_N=\{\vi_i\}_{i=1}^N$, where $\vi_1,\cdots, \vi_N$ are bounded  functions on $\Og$.
Given an integer $1\leq v\leq N$, we denote by $\mathcal{X}_v(\D_N)$ the collection of all linear spaces spanned by  $\{\vi_j:\  \ j\in J\}$  with $J\subset \{1,2,\cdots, N\}$ and $|J|=v$. A function  $f:\Og\to \CC$ is said to be $v$-sparse with respect to the dictionary $\CD_N$ if it belongs to a  linear space from the collection $\mathcal{X}_v(\D_N)$. 
We 
denote by $\Sigma_v(\D_N)$  the set of all $v$-sparse functions with respect to the dictionary $\CD_N$. In other words, we have
\begin{align*}
\Sigma_v(\D_N):&= \bigcup_{V\in\cX_v(\D_N)} V.
\end{align*}
Given $1\leq p<\infty$, we also define
\[\Sigma_v^p(\D_N):=\{ f\in \Sigma_v(\D_N):\  \ \|f\|_p\leq 1\},\  \ v=1,2,\cdots, N.\]

A  dictionary  
 $ \CD_N:=\{\ff_j\}_{j=1}^N $ is   called  uniformly bounded if 
 \be \label{I9}\sup_{x\in\Og} |\vi_j(x)|\leq 1,\   \ j=1,2,\cdots, N.\ee
 It  is called   a uniformly bounded Riesz basis if 
  (\ref{I9}) is satisfied, and there exist constants $0< R_1 \le R_2 <\infty$ such that for   any  $(a_1,\cdots, a_N)\in\CC^N$,
\begin{equation}\label{Riesz0}
R_1 \left( \sum_{j=1}^N |a_j|^2\right)^{1/2} \le \left\|\sum_{j=1}^N a_j\ff_j\right\|_2 \le R_2 \left( \sum_{j=1}^N |a_j|^2\right)^{1/2}.
\end{equation}

Next, we  give  a brief summary of recent results  on universal sampling discretization and its various applications that were obtained in \cite{DT} and \cite{DT3}.
These results give  some background information for the discussion  in the current paper. 	 

  Theorem A below is one of the main results of \cite{DT}. It  ensures  the existence of a set of at most  $Cv(\log N)^2(\log v)^2$ points which provides universal discretization of  $L_p$ norms of functions from $\Sigma_v(\CD_N)$ for $1\leq p\leq 2$ under the assumptions \eqref{I9} and \eqref{Riesz0} on the dictionary $\CD_N$.\\

{\bf Theorem A.}\cite[Theorem 1.3]{DT}\  \ {\it  Let $\CD_N=\{\vi_j\}_{j=1}^N$  be  a uniformly bounded Riesz basis satisfying \eqref{I9} and \eqref{Riesz0} for some constants $0<R_1\leq R_2$. Then 
for each  $1\le p\le 2$, and  any integer $1\leq v\leq N$,  there exists a set of 
	$m$ points
	$\xi^1,\cdots, \xi^m \in  \Omega$ with
	$
	m \le Cv(\log N)^2(\log(2v))^2
	$
	such that the inequalities 
	\begin{equation}\label{1.4u}
	\frac{1}{2}\|f\|_p^p \le \frac{1}{m}\sum_{j=1}^m |f(\xi^j)|^p \le \frac{3}{2}\|f\|_p^p
	\end{equation}
	hold for any $f\in  \Sigma_v(\D_N)$, where  $C=C(p, R_1,R_2)>1$ is a constant depending only on $p$, $R_1$ and $R_2$. \\
}

Theorem A  gives  the following  upper bound on the minimal number of points for good  universal discretization:
\begin{equation}\label{1-6} m\Bl(\mathbf{\mathcal{X}}_v(\D_N)\Br)\leq C v(\log N)^2(\log (2v))^2.\end{equation}  This upper bound   is linear
in $v$ with  extra logarithmic terms in $N$ and $v$, and  is  reasonably good since  the   lower bound $m\Bl(\mathbf{\mathcal{X}}_v(\D_N)\Br)\ge v$
 holds trivially.  
 A weaker upper bound which  is  quadratic in $v$ was previously obtained.   We  refer  to \cite{DT}, \cite{DPTT}, and \cite{KKLT} for historical comments.

In  \cite{DT3} we showed that points  that provide  good universal discretization can be used 
for sparse recovery of multivariate smooth functions. In particular, we used the universal discretization results for the $L_2$ norm 
for analyzing a theoretical sparse recovery algorithm with respect to  a special dictionary $\D_N$, and  proved  the  
Lebesgue-type inequalities for the corresponding algorithm. 
 See Section \ref{ssr} for a more detailed description of the results of  \cite{DT3}.

We now proceed with  a detailed description of the main results of this paper. Our first result gives an   improvement of Theorem A.  

\begin{Theorem}\label{IT1}  Assume that $ \CD_N=\{\ff_j\}_{j=1}^N\subset L_\infty(\Og)$ is a  dictionary of $N$ uniformly  bounded functions on $\Og$ such that  \eqref{I9} is satisfied and 
	there exists a constant $K\ge 1$ such that   for any $(a_1,\cdots, a_N) \in\bbC^N,$
	\begin{equation}\label{Riesz}
	\sum_{j=1}^N |a_j|^2 \le K  \left\|\sum_{j=1}^N a_j\ff_j\right\|^2_2 .
	\end{equation}
	Let $\xi^1,\cdots, \xi^m$ be independent 
	and identically distributed 
	random points on $\Og$  with the probability distribution  $\mu$.
	Then for  any  $1\le p\le 2$ and 
	 any   integer  $1\leq v\leq N$, the inequalities 
	\begin{equation}\label{1.4u}
	\frac{1}{2}\|f\|_p^p \le \frac{1}{m}\sum_{j=1}^m |f(\xi^j)|^p \le \frac{3}{2}\|f\|_p^p,\   \   \ \forall f\in  \Sigma_v(\CD_N)
	\end{equation}
	hold with probability $\ge 1-2 \exp\Bl( -\f {cm}{Kv\log^2 (2Kv)}\Br)$, provided that
	$$
	m \ge  C Kv \log N\cdot (\log(2Kv ))^2\cdot (\log (2Kv )+\log\log N).
	$$
	where  $C, c$  are positive constants depending only on $p$.
\end{Theorem}

Several remarks are in order.

\begin{Remark}
		It is easily seen that \eqref{Riesz} holds for some constant $K>0$ if and only if  $\vi_1, \cdots, \vi_N$ are linearly independent in $L_2(\Og, \mu)$, in which case we can choose  $K^{-1}$ to be the smallest eigenvalue of the $N\times N$ 
	matrix $\Bl[ \langle \vi_i, \vi_j\rangle_{L_2(\mu)}\Br]_{1\leq i, j\leq N}$.
	\end{Remark}

\begin{Remark}

Theorem \ref{IT1} improves Theorem A in several ways. First, the condition \eqref{Riesz} is weaker than the condition \eqref{Riesz0}, which particularly means that  estimates in 
Theorem \ref{IT1} are independent of the constant $R_2$ in  \eqref{Riesz0}. Second, Theorem \ref{IT1} gives a better  upper bound on the number of points:
\begin{equation}\label{1-9a} m\Bl(\mathbf{\mathcal{X}}_v(\D_N)\Br)\leq C v(\log N)(\log (2v))^2 (\log (2v) +\log\log N).\end{equation}
 Here we improve the dependence on $N$ from $(\log N)^2$ in \eqref{1-6}  to  
 $(\log N)( \log\log N)$   at the cost of an additional factor $\log (2v)$. In particular, we obtain from \eqref{1-9a} that  for each fixed $1\leq v\leq N$,  
 \[ m\Bl(\mathbf{\mathcal{X}}_v(\D_N)\Br)\leq C_v (\log N)(\log\log N).\]
 It is interesting to point out that this last estimate  is almost optimal  as $N\to \infty$  in the sense that there exists  a uniformly bounded Riesz basis $\D_N$ for each $N\in\NN$ such that  $$ m\Bl(\mathbf{\mathcal{X}}_v(\D_N)\Br)\ge c \log N,\  \ v=1,2,\cdots, N,$$
 where $c>0$ is an absolute constant.   See Section \ref{lb} for details. 
 Finally, Theorem A only ensures  the existence of a set of good  points which provides  universal discretization, while Theorem \ref{IT1} shows that  a set of  independent random points that are  identically distributed  according to $\mu$ provides good  universal discretization with high probability.  The step from the existence result  to  the statement ``with high probability for randomly chosen points" turns out to be nontrivial.

	\end{Remark}

\begin{Remark}
	The conclusion of Theorem \ref{IT1}  remains true for each integer $1\leq v\leq N$ if we replace 
	the condition \eqref{Riesz} with the following weaker one:  
	 for any $(a_1,\cdots, a_N) \in\bbC^N$, 
	$$
	\sum_{j\in J} |a_j|^2 \le K  \left\|\sum_{j\in J} a_j\ff_j\right\|^2_2,\   \   \ \forall J\in\mathcal{J}_v,\eqno{(1.7\textnormal{a})}
	$$
	where $\mathcal{J}_v$ denotes the collection of all subsets $J$ of $\{1,2,\cdots, N\}$ with $|J|=v$. 
	
\end{Remark}

Theorem \ref{IT1} has interesting  applications in  sparse sampling recovery.  As our next result shows,   a simple greedy algorithm - Weak Orthogonal Matching Pursuit (WOMP) - can effectively recover sparse functions  by utilizing points in Theorem \ref{IT1} that provide good universal discretization. To formulate this result, we need to  describe   some necessary  notations.

Let $X=(X, \|\cdot\|)$ be a Banach space of functions on $\Og$.
Let $\CD_N=\{\vi_i\}_{i=1}^N\subset X$ be a dictionary in $X$.  
For $v=1,2,\cdots$, define 
$$
\sigma_v(f,\D_N)_X := \inf_{g\in\Sigma_v(\D_N)}\|f-g\|
$$
to be the best $v$-term approximation of  $f\in X$   in the norm  of $X$ with respect to $\D_N$.
Furthermore,  for a function class $\bF\subset X$,  we define 
$$
\sigma_v(\bF,\D_N)_X := \sup_{f\in\bF} \sigma_v(f,\D_N)_X,\quad   v=1,2,\cdots$$
and $$  \sigma_0(\bF,\D_N)_X := \sup_{f\in\bF} \|f\|.
$$

The  Weak Orthogonal Matching Pursuit (WOMP)  is a  greedy algorithm defined with respect to a given dictionary $\D_N=\{\vi_i\}_{i=1}^N$  in a  Hilbert space of functions on $\Og$   equipped with the inner product $ \<\cdot,\cdot\>$ and the norm  $\|\cdot\|$.   It was also 
defined in \cite{T1} under the name Weak Orthogonal Greedy Algorithm.\\

{\bf Weak Orthogonal Matching Pursuit (WOMP).} Let $\D_N=\{\vi_i\}_{i\in I}$ be a dictionary of countably many   nonzero elements in $H$, where $I\subset \NN$.
 Let $\tau:=\{t_k\}_{k=1}^\infty\subset (0, 1]$ be a given  sequence of weakness parameters. 
Given  $f_0\in H$, we define  a sequence  $\{f_k\}_{k=1}^\infty\subset H$ of functions  for $k=1,2,\cdots$  inductively  as follows: 
\begin{enumerate}[\rm (1)]
	\item 
 $ j_k\in I$  is any  integer from the index  set $I$  satisfying
$$
|\langle f_{k-1},\psi_{j_k}\rangle | \ge t_k
\sup_{i\in I} |\langle f_{k-1},\psi_i\rangle |,
$$
where $\psi_i=\vi_i/\|\vi_i\|$ for $i\in I$.

\item  Let  $H_k := \sp \{\vi_{j_1},\dots,\vi_{j_k}\}$, and define 
$G_k(\cdot, \CD)$ to be the orthogonal projection operator from $H$   onto the space $H_k$ .

\item   Define the residual after the $k$th iteration of the algorithm by
\begin{equation*}
f_k := f_0-G_k(f_0, \D).
\end{equation*}
\end{enumerate}

In the case when $t_k=1$ for $k=1,2,\dots$,   WOMP is called the Orthogonal
Matching Pursuit (OMP). In this paper we only consider the case when  $t_k=t\in (0, 1]$ for $k=1,2,\dots$.  The term {\it weak} in the definition of the WOMP means that at step (1) we do not shoot for the optimal element of the dictionary which attains the corresponding supremum. The obvious reason for this is that we do not know in general if the optimal element exists. Another practical reason is that the weaker the assumption is, the easier it is to realize in practice. Clearly, $\vi_{j_k}$ may not be unique.  However, all the  results formulated  below  are independent of the choice of the $\vi_{j_k}$.

For the sake of convenience in later applications, we will use the notation  $\text{WOMP}\bigl(\D; t\bigr)_H$ to   denote the WOMP   defined with respect to  a given weakness parameter $t\in (0, 1]$ and   a dictionary $\CD$ in a Hilbert  space $H$.

Now we are in a position to formulate our second theorem. We will  consider the Hilbert space $L_2(\Omega_m,\mu_m)$ 
 instead of  $L_2(\Omega,\mu)$,  
where $\Omega_m=\{\xi^\nu\}_{\nu=1}^m$ is a set of points that provides a good universal discretization,  and $\mu_m$ is the uniform probability measure on $\Og_m$, i.e., 
$\mu_m\{ \xi^\nu\} =1/m$, $\nu=1,\dots,m$.    Let $\CD_N(\Omega_m)$ denote  the restriction 
of a dictionary  $\CD_N$ on the set  $\Omega_m$.
Theorem \ref{IT2} below  guarantees that the simple greedy algorithm WOMP gives the corresponding Lebesgue-type inequality in the norm $L_2(\Omega_m,\mu_m)$, and hence 
 provides good sparse recovery. 

\begin{Theorem}\label{IT2}  Let  $\CD_N=\{\vi_j\}_{j=1}^N$ be  a uniformly bounded Riesz basis  in $L_2(\Og, \mu)$  satisfying  \eqref{I9} and \eqref{Riesz0} for some constants $0<R_1\leq R_2<\infty$. 
Let $\Og_m=\{\xi^1,\cdots, \xi^m\}$ be a finite subset of $\Og$  that provides   universal discretization for the collection 
$\cX_u(\CD_N)$ and a given integer $1\leq u\leq N$.   Given a weakness parameter $0<t\leq 1$,   there exists a constant integer  $c=c(t,R_1,R_2)\ge 1$  such that for any integer $0\leq v\leq u/(1+c)$ and any $f_0\in L_\infty(\Omega)$,   the  $$\text{WOMP}\Bl(\D_N(\Og_m); t\Br)_{L_2(\Omega_m,\mu_m)}$$    applied to $f_0$  gives 
\be\label{mp}
\|f_{cv}    \|_{L_2(\Omega_m,\mu_m)} \le C\sigma_v(f_0,\CD_N(\Omega_m))_{L_2(\Omega_m,\mu_m)}, 
\ee
and
\be\label{mp2}
\|f_{c v}\|_{L_2(\Omega,\mu)} \le C\sigma_v(f_0,\CD_N)_\infty,
\ee
where $C>1$ is an  absolute constant, and $f_k$ denotes the residue of $f_0$ after the $k$-th iteration of the algorithm.
 \end{Theorem}

 \begin{Corollary}\label{IC1} Under the  assumptions of  Theorem \ref{IT2},  the $$\text{WOMP}\Bl(\D_N(\Og_m); t\Br)_{L_2(\Omega_m,\mu_m)}$$  recovers every $v$-sparse $f_0\in \Sigma_v(\CD_N)$ exactly for any positive integer $v$ with $(1+c)v\leq u$. 
\end{Corollary}
     
 Theorem \ref{IT2} has an interesting application in   optimal sampling recovery. 
 For a function class $\bF\subset \cC(\Omega)$,  we define
$$
\varrho_m^o(\bF,L_p) := \inf_{\Psi; \xi } \sup_{f\in \bF}\|f-\Psi(f(\xi^1),\dots,f(\xi^m))\|_p,
$$
where the infimum is taken  over all  mappings $\Psi : \bbC^m \to   L_p(\Omega,\mu)$, and all subsets $\xi=\{\xi^\nu\}_{\nu=1}^m\subset \Og$ of $m$ points $\xi^1,\cdots, \xi^m\in\Og$.
Here, the superscript {\it o} stands for optimal. 
The reader can find some recent results  on   optimal sampling recovery in Section \ref{ssr}.
As  a direct corollary of 
Theorem \ref{IT2}, we have:

\begin{Corollary}\label{IT3}  Let  $\CD_N$ be  a uniformly bounded Riesz basis  in $L_2(\Og, \mu)$  satisfying  \eqref{I9} and \eqref{Riesz0} for some constants $0<R_1\leq R_2<\infty$.  Then  there exists a constant  $c=C(R_1,R_2)\ge 1$  depending only on  $R_1$ and $R_2$  such that  for any function class $\bF\subset \cC(\Omega)$, and  any positive integers $u$ and $m$ with  $1+c\leq u\leq N$ and $m\ge m(\cX_u(\CD_N))$, we have 
\be\label{or}
\varrho_m^o(\bF,L_2) \le C\sigma_{\floor* {u/(1+c)}}(\bF,\CD_N)_\infty,
\ee
where $C>1$ is an absolute constant.
 \end{Corollary}

The first inequality of the kind of (\ref{or}) was obtained in the  recent paper \cite{JUV} by T. Jahn, T. Ullrich, and F. Voigtlaender, where they used deep known results from compressed sensing and applied 
the $\ell_1$ minimization algorithm. In addition to function values at $m$ points their algorithm uses 
the fact $f\in \bF$ and the quantity $\sigma_v(\bF,\CD_N)_\infty$. They proved an inequality similar to (\ref{or}) for a uniformly bounded orthonormal system $\CD_N$ with an extra term (not important one) in the right hand side of (\ref{or}). Their algorithm does not provide a sparse with respect to $\CD_N$ approximant. 
We improve the result from \cite{JUV} in a number of ways. First, we get rid of the extra term in the right hand side. Second, our algorithm (WOMP) only uses the function values at $m$ points and does not 
use any information about the class $\bF$. Thus, it is universal in that sense. Third, we prove (\ref{or}) for a wider class of systems $\CD_N$. Forth, our algorithm provides $cv$-sparse with respect 
to $\CD_N$ approximant. Finally, for our algorithm we prove the Lebesgue-type inequality (\ref{mp2}) for individual functions.  

The rest of this paper is organized  as follows.  In Section \ref{X}, we present a conditional theorem, Theorem \ref{XT2}, without proof, which allows us to   establish  universal sampling discretization of the $L_p$ norms with randomly selected sampling points  by utilizing certain  entropy estimates of $\Sigma_v^p(\D_N)$ in the uniform norm. This theorem  is a direct consequence of a more general conditional theorem, Theorem \ref{thm-4-2}, which we prove in Section \ref{sec:5}. We also give several useful estimates on the entropy numbers of the set $\Sigma_v^p(\D_N)$ in the uniform norm in Section \ref{X}, following closely the techniques developed in \cite{DT}.
In Section \ref{Z}, we prove Theorem \ref{IT1} by utilizing the conditional Theorem \ref{XT2} and the entropy number estimates presented in Section \ref{X}. 

After that, in Section \ref{lb} 
we provide an example of a  uniformly bounded Riesz system 
  $\CD=\{\vi_j\}_{j=1}^\infty$  
such  that for  $\D_N=\{\vi_j\}_{j=1}^N$ with $N=1,2,\cdots$, we have
$$  m(\mathbf{\mathcal{X}}_v(\D_N))\ge c \log N,  \  \   v=1,2,\cdots, N, $$
 for some absolute constant $c>0$, meaning that universal sampling discretization for the collection 
 $\mathcal{X}_v(\D_N)$
 normally requires at least $C_v \log N$ points.  
We dedicate Section \ref{ssr} to discussing  applications of universal sampling discretization in sparse sampling recovery. In particular,    Theorem \ref{IT2} and  several other related results  are proved in Section \ref{ssr}.  Finally, in Section \ref{sec:5}, we prove a more general version of the conditional theorem stated in Section \ref{X}, which  corresponds to  the random version of Theorem 5.1 of \cite{DT}.

 \section{A conditional theorem on universal sampling discretization}
\label{X}

We  first recall the definition of entropy numbers in Banach spaces.   Let $A$ be  a compact subset  of a Banach space $(X, \|\cdot\|)$. 
Given a positive number $\e$,   the covering number $N_\e(A, X)$ of  the  set $A$ in the space $X$  is defined  to be the smallest positive integer $n$ for which there exist $n$ elements $g_1, \cdots, g_n\in A$ such that 
\[ \max_{f\in A} \min_{1\leq j\leq n}\|f-g_j\|\leq \va.\]
We denote by $\cN_\e(A,X)$
the corresponding minimal $\e$-net of the set $A$ in $X$.
The $\va$-entropy $\mathcal{H}_\va (A, X)$ of the compact set $A$  in $X$  is then  defined as
$$\mathcal{H}_\va (A, X)=\log_2 N_\e(A,X), \   \ \va>0,$$
whereas   the entropy numbers  of $A$ in the space $X$ are defined as 
$$
\e_k(A,X) := \inf\{t>0: \   \ \cH_t(A, X)\leq k\},\   \ k=0, 1,2,\cdots.
$$

The following   theorem  
  is a direct consequence of  a  more general conditional theorem that is proved in Section \ref{sec:5} (see Theorem \ref{thm-4-2}).

 \begin{Theorem}\label{XT2} Let $1\le q<\infty$ and $1\leq v\leq N$. Suppose that  $\D_N$ is a dictionary of $N$ functions from $L_\infty(\Og)$  such that the set 
 	$$
 	\Sigma_v^q(\D_N):= \{f\in \Sigma_v(\D_N)\,:\, \|f\|_q\le 1\}
 	$$
 	 satisfies 
\be\label{Z1}
\e_k(\Sigma_v^q(\D_N),L_\infty) \le  B_1 (v/k)^{1/q}, \quad k=1,2,\cdots,
\ee
and 
\begin{equation}\label{2-3b} 
\|f\|_\infty \leq B_2 v^{1/q} \|f\|_q,\   \   \forall f\in \Sigma_v(\D_N)
\end{equation}
for some constants $B_1, B_2\ge 1$. 
 	Let $\xi^1, \cdots, \xi^m\in\Og$ be independent random points   satisfying \begin{equation}\label{2-3}
	\f 1m \sum_{k=1}^m \mu_{\xi^k}=\mu,
	\end{equation}
	 where $\mu_{\xi^j}$ denotes the probability distribution of $\xi^j$. 
		If 
	$$
	m \ge  C(q)     B_1^{q} v (\log (2B_2v))^2,
	$$
	for some large constant $C(q)>1$  depending on $q$,
then the inequalities 
	$$
\f34\|f\|_q^q \le \frac{1}{m}\sum_{j=1}^m |f(\xi^j)|^q \le \f54 \|f\|_q^q,\   \ \forall f\in  \Sigma_v(\D_N)
$$
hold 
 with probability 
	\[ \ge 1- \exp\Bl[ -\f { c_q m }{B_2^q v \log^2 ( 2B_2 v)}\Br]. \]
	\end{Theorem}

Clearly, 
 either of  the following two conditions implies the condition \eqref{2-3} in the above theorem:
\begin{enumerate}[\rm (i)]
	\item  $\xi^1,\cdots,\xi^m$ are  identically distributed according to  $\mu$;
	\item there exists a partition $\{\Ld_1, \cdots, \Ld_m\}$ of $\Og$ such that $\mu(\Ld_j)=\f1m$ and $\xi^j\in\Ld_j$ is  distributed  according to    $m\cdot \mu\Bl|_{\Ld_j}$  for each $1\leq j\leq m$. 
\end{enumerate}

Next, we recall the following simple remark from \cite{DT}.

\begin{Remark}
	We point out that \eqref{Z1} implies   
	\be\label{2.3a} \|f\|_\infty \leq 3 B_1 v^{1/q},\   \  \forall f\in \Sigma^q_v(\D_N).
	\ee
Therefore, assumption (\ref{2-3b}) can be dropped with $B_2$ replaced by $3B_1$ in the bound on $m$.	 However, in  applications, the constant $B_2$ in \eqref{2-3b} may be significantly smaller than  $3B_1$.  For example, if $\D_N$ is a uniformly bounded orthonormal system with $\max_{f\in\D_N}\|f\|_\infty=1$, then we can take $B_2=1$ in the case $q=2$.  
\end{Remark}	
	
%
%
%

Theorem \ref{XT2}    motivates us to estimate the quantities  $\e_k(\Sigma_v^q(\D_N),L_\infty)$. 
 Our goal for the reminder of this section is to  prove that if  the  dictionary $\CD_N$ satisfies both  (\ref{I9}) and (\textnormal{(1.7a)}), and if
 there exist a number $q_{ v}:= q_{N,v}>2$ and a universal constant $C_0>1$ such that 
 \be\label{X5a}
 \|f\|_\infty \le C_0\|f\|_{q_{v}},\   \ \forall f\in \Sigma_{2v}(\CD_N), 
 \ee
 then we have 
\be
\e_k(\Sigma_v^2(\CD_N),L_\infty) \le C{q_v} \Bigl(\frac vk\Bigr)^{1/2},\quad   k=1,2,\dots.\label{1-4b}
\ee

We now recall
some known general results, which will be needed in our later proof.
\begin{Theorem}\label{XT3} \textnormal{  (\cite{VT138}, \cite[p.331, Theorem 7.4.3]{VTbookMA})}. Let  $W$ be  a compact subset of a Banach space $X$ for which  there exist a  system $\D_N\subset X$ with  $|\D_N|=N$, and  a number $r>0$ such that 
$$
  \sigma_m(W,\D_N)_X \le (m+1)^{-r},\quad m=0,1,\cdots, N.
$$
Then we have
\begin{equation}\label{X3}
\e_k(W,X) \le C(r) \left(\frac{\log(2N/k)}{k}\right)^r,\  \ k=1,\cdots, N.
\end{equation}
\end{Theorem}

For a given set $\D_N=\{g_j\}_{j=1}^N$ of $N$ elements in  a Banach space $X$,
 we introduce the  generalized octahedron,
\be\label{X4}
A_1(\D_N) := \left\{f\,:\, f=\sum_{j=1}^N c_jg_j,\quad \sum_{j=1}^N |c_j|\le 1\right\},
\ee
and the norm $\|\cdot\|_A$ on $X_N=\spn\{g_1,\cdots, g_N\}$,
$$
\|f\|_A := \inf\left\{ \sum_{j=1}^N |c_j|\,:\, f=\sum_{j=1}^N c_jg_j,\  \ c_j\in\CC \right\},\  \ f\in X_N.
$$

We now use a known general result for a smooth Banach space. For a Banach space $X$ we define the modulus of smoothness
$$
\rho(X,u) := \sup_{\|x\|=\|y\|=1}\left(\frac{1}{2}(\|x+uy\|+\|x-uy\|)-1\right),\  \ u>0.
$$
The uniformly smooth Banach space is the one with the property
$$
\lim_{u\to 0+}\f{\rho(X, u)}u =0.
$$
A  Banach space $X$ is called $s$-smooth  for some  parameter  $1<s\leq 2$  and constant $\ga>0$ if  $\rho(X, u) \le \ga u^s$ for all $u>0$. The following bound is a corollary of known greedy approximation results (see, for instance \cite{VTbookMA}, p.455).

 \begin{Theorem}\label{XT4}\textnormal{(\cite[p. 455]{VTbookMA}) } Let $X$ be an  $s$-smooth Banach space  for some parameter $1<s\leq 2$ and constant $\ga>0$.
 	Let $\D_N=\{g_j\}_{j=1}^N\subset X$ be a dictionary normalized by $\|g_j\|_X= 1$ for all $1\leq j\leq N$. 
 Then for any $1\leq m\leq N$, 
 $$
 \sigma_m(A_1(\D_N), \D_N)_X \le C(s)\gamma^{1/s}m^{1/s-1}.
 $$
 \end{Theorem}

 We now proceed to a special case of  $X=L_p$, where  it is known that 
 \be\label{X5-0}
 \rho(L_p,u) \le (p-1)u^2/2,\quad 2\le p<\infty,\  \ u>0.
 \ee

\begin{Theorem}\label{XT5}   
	 Let $\CD_N=\{\vi_j\}_{j=1}^N \subset L_\infty(\Og)$ be  a  dictionary  satisfying  the conditions  \eqref{I9},   \textnormal{(1.7a)}
	 and  \eqref{X5a} for some constants $K\ge 1$ and  $q_{ v}:= q_{N,v}>2$. 
 Then   we have
 \be\label{X5}
 \e_k(\Sigma_v^2(\CD_N),L_\infty) \le C   \sqrt{K q_{v} \cdot \log N} \Bigl(\frac vk\Bigr)^{1/2},\quad   k=1,2,\dots.
 \ee
 \end{Theorem}
 \begin{proof} First of all, for any $f=\sum_{j\in G }a_j\ff_j$ with  $|G|=v$, we have 
  \be\label{X6}
 \|f\|_A \le \sum_{j\in G} |a_j| \le  v^{1/2} \left( \sum_{j\in G} |a_j|^2\right)^{1/2} \le (Kv)^{1/2}\|f\|_2,
 \ee
 implying 
 $$
 \Sigma_v^2(\CD_N) \subset (Kv)^{1/2}\Sigma_v^A(\CD_N),$$
 where $$ R\Sigma_v^A(\CD_N):=\{f\in\Sigma_v(\CD_N)\,:\, \|f\|_A \le R\}.
 $$
 By Theorem \ref{XT4} with $s=2$ and by (\ref{X5-0}), we have that 
 for $p\in[2,\infty)$
 \be\label{X7-0}
 \sigma_m(\Sigma_v^2(\CD_N),\CD_N)_{L_p} \le C  (K v)^{1/2} \sqrt p m^{-\frac12},\quad   m=1,2,\cdots,N.
 \ee
 Thus,   Theorem \ref{XT3} implies that  for $p\in[2,\infty)$
 \be\label{X7}
  \e_k(\Sigma_v^2(\CD_N),L_p) \le C  (K p \log(2N/k))^{1/2}(v/k)^{1/2},\quad   k=1,2,\dots,N.
 \ee
 Second, by (\ref{X5a})  we obtain  
 \be\label{X8}
   \e_k(\Sigma_v^2(\CD_N),L_\infty)\le C_0\e_k(\Sigma_v^2(\CD_N),L_{q_{v} }). 
 \ee 
 Combining (\ref{X7}) and (\ref{X8}) we get   
 \be\label{X9}
  \e_k(\Sigma_v^2(\CD_N),L_\infty) \le C   \sqrt{Kq_v \cdot \log N}  (v/k)^{1/2},\quad   k=1,2,\dots,N.
 \ee

 Finally, for $k>N$ we use the inequalities (see \cite[p. 323, (7.1.6)]{VTbookMA} and \cite[p. 324, Cor 7.2.2]{VTbookMA})
 $$
 \e_k(W,L_\infty) \le \e_N(W,L_\infty)\e_{k-N}(X_N^\infty,L_\infty)
 $$
 and 
 $$
 \e_{n}(X_N^\infty,L_\infty) \le 3\cdot 2^{-n/N},\qquad 2^{-x} \le 1/x,\quad x\ge 1,
 $$
 to obtain (\ref{X9}) for all $k$.
 This completes the proof. 
 
 \end{proof}
 
In the same way as in Section 3 of \cite{DT} we derive from Theorem \ref{XT5}, which holds for $p=2$, the following Theorem \ref{YT1}, which holds for $1\le p\le 2$. We do not present this proof here.

\begin{Theorem}\label{YT1}   
	 Assume that  $\CD_N=\{\vi_j\}_{j=1}^N \subset L_\infty(\Og)$ is   a  dictionary  satisfying  the conditions  \eqref{I9},   \textnormal{(1.7a)}
	and  \eqref{X5a} for some constants $K\ge 1$ and  $q_{ v}:= q_{N,v}>2$.  Then  for $1\leq p\leq 2$,  we have
 \be\label{X5}
 \e_k(\Sigma_v^p(\CD_N),L_\infty) \le C_p  \Bl( \f{K q_v \cdot \log N\cdot v}k \Br)^{1/p},\quad   k=1,2,\dots.
 \ee
 \end{Theorem}

\section{Proof of Theorem \ref{IT1}}\label{Z}

 Theorem \ref{YT1}  provides bounds on the entropy numbers $ \e_k(\Sigma_v^p(\CD_N),L_\infty)$ under the additional assumption \eqref{X5a}
$$
\|f\|_\infty \le C_0\|f\|_{q_{v}},\   \ \forall f\in \Sigma_{2v}(\CD_N).
$$
  Thus, a combination of Theorem \ref{YT1} with Theorem \ref{XT2} implies  the conclusion of Theorem \ref{IT1} 
   for
   $$
   m \ge  C(p)     ( Kq_v\cdot \log N)  v (\log (2Kv))^2
   $$
   under the additional condition \eqref{X5a}, which is  not assumed in Theorem ~\ref{IT1}.  In this section, we will show how to drop  this additional assumption.

Recall that  $\CD_N:=\{\ff_j\}_{j=1}^N\subset L_\infty(\Og)$  is  a dictionary  satisfying the conditions \eqref{I9} and  \eqref{Riesz}  for some constant   $K\ge 1$.
For a set $J\subset [N]:= \{1, 2,\cdots, N\}$, we define
\[ V_J:=\text{span}\ \Bl\{ \varphi_j:\  \  j\in J\Br\}\   \ \text{and}\  \  V_J^p:=\{f\in V_J:\  \ \|f\|_p\leq 1\}.\]
Then 
$\Sigma_v(\CD_N):=\underset{\sub{J\subset [N]\\
		|J|=v}}{\bigcup} V_J$ for $1\leq v\leq N$. 
Furthermore, by  \eqref{I9} and \textnormal{(1.7a)}, we have 
\[\|f\|_\infty \leq  (K v)^{\frac12} \|f\|_2,\   \ \forall f\in \Sigma_v(\CD_N),\]
which in turn implies that 
\begin{equation}\label{1-2:eq}\|f\|_\infty \leq  (Kv) ^{\frac1p} \|f\|_p,\   \ \forall f\in \Sigma_v(\CD_N),\   \ 1\leq p\leq 2.\end{equation}

We need several lemmas. The  first three lemmas can be found in \cite{BLM}.

\begin{Lemma}\label{lem-1-3} 	\textnormal{\cite[Lemma 2.4]{BLM}}     Let  $X$ denote  the space $\RR^N$ endowed with  a  norm $\|\cdot\|_X$, and let $B_X:=\{x\in X:\  \|x\|_X \leq 1\}$. 
	Then
 for any $0<\va\leq 1$,
		\begin{equation}\label{3-2-b}
		N\log \f 1\va \leq \mathcal{H}_\va (B_X, X) \leq N\log (1+\f 2\va).
		\end{equation}
	
\end{Lemma}

\begin{Lemma}\label{lem-2-2}
	{ \textnormal{\cite[Lemma 2.1]{BLM}} }
	Let $\{g_j\}_{j=1}^m$ be independent random variables with mean $0$ on some probability space $(\Og, \mu)$, which satisfy$$\max_{1\leq j\leq m} \|g_j\|_{L_1(d\mu)}\leq M_1, \  \  \max_{1\leq j\leq m}  \|g_j\|_{L_\infty(d\mu)} \leq M_\infty$$
	for some constants $M_1$ and $M_\infty$. Then for any $0<\va<1$ we have the inequality
	\begin{equation*}
\mu\Bl\{ \og\in\Og:\  \ 	\Bl|\f 1m \sum_{j=1}^m g_j(\og)\Br| \ge \va\Br\} \leq 2 e^{-\f{m\va^2}{ 4M_1M_\infty}}.
	\end{equation*}

\end{Lemma}

\begin{Lemma}\label{lem-1-2}
	{ \textnormal{\cite[Lemma 2.5]{BLM}} }
	Let $T: X\to Y$ be a bounded linear map from a normed linear space $(X, \|\cdot\|_X)$ into another normed linear  space $(Y,\|\cdot\|_Y)$.   Let  $\mathcal{F}$ be an $\va$-net of the unit ball $B_X:=\{x\in X:\  \|x\|_X\leq 1\}$ for some constant $\va\in (0, 1)$. Assume that there exist constants $C_1, C_2>0$ such that
	$$ C_1 \|x\|_X \leq \|T x\|_Y\leq C_2 \|x\|_X,\   \   \  \forall x\in \mathcal{F}.$$
	Then
	$$ C_1(\va)  \|z\|_X\leq \|Tz\|_Y\leq C_2(\va) \|z\|_X,\   \ \forall z\in X,$$
	where
	\begin{align*}
	C_1(\va):= &C_1(1-\va)-   C_2\va \f {1+\va} {1-\va},\  \
	C_2(\va):=  {C_2 }\f {1+\va}{1-\va}.
	\end{align*}
\end{Lemma}

From the above three lemmas, we can deduce
\begin{Lemma}\label{lem-3-4}   Let $1\leq p\leq 2$  be a fixed number. 
	Let $\{\xi^j\}_{j=1}^\infty$ be a sequence of independent random points identically  distributed according to  $\mu$.
	Then for any   $0< \va\leq \frac 18$ and  any integer   \begin{equation}
	 m\ge 16 K\va^{-2} v^2 \log \Bl( \f {4e} {\va} \cdot \f {N} {v} \Br),\label{4-3-a}
	\end{equation}  the   inequalities 
	\begin{align}\label{4-3-b}
	(1-4\va) \|f\|_p^p\leq  \frac 1m \sum_{j=1}^m| f(\xi^j)|^p \leq (1+4\va)\|f\|_p^p,\   \ \forall f\in \Sigma_v(\CD_N)
	\end{align}hold  with probability  $\ge 1- 2\exp\Bl(-\f {m\va^2}{ 16 K v}\Br)$. 
	\end{Lemma}

\begin{proof}   By Lemma \ref{lem-1-3},	
	for each $J\subset [N]$ with $|J|=v$, 
 there exists  an $\va$-net    $\FF_J\subset V_J^p$  of $V_J^p$ in the space $L_p$  such that
	$|\FF_J| \leq \Bl(1+\f 2\va\Br)^v.$
	Let $\FF:=\bigcup_{J\subset [N],\  |J|=v} \FF_J.$
	Then 
	\begin{align}
	\log |\FF| &\leq \log\Bl[\binom{N}v \Bl(1+\f 2\va\Br)^v\Br]\leq \log\Bl[\Bl( \f{eN} v\Br)^v \Bl(1+\f 2\va\Br)^v\Br]\notag\\
	&\leq v \log \Bl( \f {4eN} {v\va} \Br).\label{1-5:eq}
	\end{align}

	On the other hand, 	 \eqref{1-2:eq} implies 
	\[ \|f\|_\infty^p \leq Kv,\   \ \forall f\in \FF.\]
 Thus,  using \eqref{1-5:eq} and  Lemma \ref{lem-2-2}, we  conclude   that the inequalities
	\begin{equation}\label{7-4-0} \Bl| \f 1 m \sum_{j=1}^m| f(\xi^j) |^p- \int_{\Og} |f|^p\, d\mu\Br| \leq \va\|f\|_p^p,\   \ \forall f\in \FF\end{equation}
	hold  with probability
	\begin{align*}
	\ge &1-2|\FF|\exp\Bl(- \f {m \va^2} {8Kv}\Br)=1-2\exp\Bl( \log |\FF|- \f {m \va^2} {8Kv}\Br)\\
	\ge &1-2\exp\Bl(v \log \Bl( \f {4eN} {v\va} \Br)-\f {m \va^2} {8Kv}\Br).
	\end{align*}
	Under the condition \eqref{4-3-a}, we have 
	\[ \f {m \va^2} {16Kv}\ge v \log \Bl( \f {4eN} {v\va} \Br), \]
implying that  \eqref{7-4-0} holds with probability
	\[ \ge 1- 2\exp\Bl(-\f {m \va^2} {16Kv}\Br).\]
	To complete the proof, we just need to note that by Lemma \ref{lem-1-2}, \eqref{7-4-0} implies \eqref{4-3-b}.

\end{proof}

Given $\bx=(x_1, x_2,\cdots, x_m)\in\Og^m$ and $f:\Og\to \RR$, we define 
\[ S(f, \bx):=(f(x_1), f(x_2), \cdots, f(x_m))\in\RR^m.\]
 We denote by  $L_p^{m}$  the space $\RR^m$ equipped with the discrete norm 
$$\|\bz\|_p :=\begin{cases}
(\f 1{m}  \sum_{j=1}^{m} |z_j|^p\Br)^{1/p},&\   \  1\leq p<\infty,\\
\underset{1\leq j\leq m} {\max} |z_j|,&\   \ p=\infty,
\end{cases},\   \ \bz=(z_1,\cdots, z_{m})\in\R^{m}.$$
Thus, 
\[ \|S(f, \bx)\|_p =\Bl( \f 1m \sum_{j=1}^m |f(x_j)|^p \Br)^{1/p},\    \    \ 1\leq p<\infty,\]
and \[\|S(f, \bx)\|_\infty =\max_{1\leq j\leq m} |f(x_j)|.\]
As usual, 	we identify  a vector in $\RR^m$  with a function on the set $[m]:=\{1,2,\cdots, m\}$. Under this identification, $\sp\{ S(\vi_i, \bx):\  \ 1\leq i\leq N\}$ is a subspace of $L_p^m$. 


Now we are in a position to prove   Theorem \ref{IT1}. 
Without loss of generality, we may assume that  $m\leq C K v^2 \log \Bl(  \f {eN} {v} \Br)$ for some large constant $C$, since otherwise  Theorem \ref{IT1} follows directly from Lemma \ref{lem-3-4}.   Let $\ell_v$ denote the smallest integer $\ge C K v^2 \log \Bl(  \f {eN} {v} \Br)$. 
Let  $m_1= m\ell_v $. 	Let $x_1, \cdots, x_{m_1}\in\Og$ be independent random points that are identically distributed according to  $\mu$. Let $\bx=(x_1, \cdots, x_{m_1})$. 
	By Lemma \ref{lem-3-4}, the inequalities 
	\begin{align}
&	\frac 45 \|f\|_p^p\leq \|S(f, \bx)\|_{p}^p \leq \frac 65\|f\|_p^p,\   \ \forall f\in\Sigma_v(\CD_N), \label{1-8}\\
	 \  \ \text{and}&\   \  	\frac 45 \|f\|_2^2\leq \|S(f, \bx)\|_{2}^2 \leq \frac 65\|f\|_2^2,\  \   \forall f\in\Sigma_v(\CD_N) \label{1-9}
	\end{align}
	hold simultaneously  with probability $$\ge 1-4\exp \Bl(-\f {cm_1}{K v}\Br)\ge 1-4e^{-mv}$$
	for some absolute constant  $c\in (0, 1)$, provided that the constant $C$ is large enough.

	Next,   we  fix the random points $x_1, \cdots, x_{m_1}$ such that  both \eqref{1-8} and \eqref{1-9}  hold.  
	Let $\{ \Ld_1,\cdots, \Ld_m\}$ be a partition of the set $[m_1]$ such that $|\Ld_j|=\ell_v$ for $1\leq j\leq m$. 
	Let $n_j\in\Ld_j$ denote a random variable  that is uniformly distributed in the finite set $\Ld_j$. 
	Assume in addition that   $x_1,\cdots, x_{m_1}, n_1,\cdots, n_m$ are independent. Define 
 $\xi^j =x_{n_j}$ for $j=1,2,\cdots, m$. It is easily seen that $\xi^1, \cdots, \xi^m\in\Og$ are independent random points  identically distributed according to  $\mu$.

	Now  we   consider the discrete  $L_p^{m_1}$ -norm    instead of the initial norm \newline $\|\cdot\|_{L_p(\Omega)}$. Let 
	\[ S(\CD_N,\bx):=\Bl\{ S(\vi_i, \bx):\  \ 1\leq i\leq N\Br\}\subset L_2^{m_1}.\]
		 \eqref{1-8} and \eqref{1-9}  imply that   $S(\CD_N,\bx)$ is a dictionary in  $L_2^{m_1}$ satisfying \eqref{I9} and \textnormal{(1.7a)} with constant $\f54K$ for the normalized counting measure on $[m_1]$. Moreover, by \eqref{1-2:eq} and \eqref{1-8}, we have
$$ \|S(f,\bx)\|_{\infty} \leq   (\f 54 Kv)^{1/p} \|S(f, \bx)\|_{p},\   \  \forall f\in\Sigma_v(\CD_N).$$
Note that  $$\log m_1 \leq  q_v:=C \log (2Kv)+ \log \log N.$$   Thus,  the regular Nikolskii's  inequality of  
the discrete  norms $L_q^{m_1}$ implies 
\[ \|S(f, \bx)\|_{\infty} \leq e\cdot  \|S(f, \bx)\|_{q_v},\   \   \forall f\in \Sigma_v(\CD_N).\]
By Theorem \ref{XT2} and  Theorem~\ref{YT1} applied to the discrete norm   of $L_p^{m_1}$,  we conclude that for any integer 
\[m\ge  C_p K v (\log N )  (\log (2vK))^2 (\log (2vK) +\log\log N), \]
the inequalities 
\begin{equation}\label{4.5}
\frac 34 \|S(f, \bx)\|_{p}^p \leq \|S(f, \xi)\|_{p}^p \leq \frac 54 \|S(f,\bx)\|_{p}^p,\  \ \forall f\in\Sigma_v(\CD_N)
\end{equation}
hold   with probability 
	\[\geq 1-  2  \exp\Bl(-\f {cm} {K v \log^2 (2Kv)}\Br),\]
where  $$\xi=(\xi^1, \cdots, \xi^m)=(x_{n_1}, x_{n_2},\cdots, x_{n_m}).$$
Combining \eqref{4.5} with \eqref{1-8}, we obtain 
\begin{equation}\label{4-10a}
\frac 35 \|f\|_{p}^p \leq \|S(f, \xi)\|_{p}^p \leq \frac 32 \|f\|_{p}^p,\  \ \forall f\in\Sigma_v(\CD_N).
\end{equation}

Finally,  we estimate the probability of the event  $E$  that \eqref{4-10a} holds.
Let $E_1$ denote the event that  both  \eqref{1-8} and \eqref{1-9}  hold.  The above argument then shows that 
$\PP(E_1) \ge 1-4 e^{-m v}$ and
$$ \chi_{E_1} \cdot \EE[ \chi_E |x_1, \cdots, x_{m_1}]= \chi_{E_1} \cdot \EE_{n_1, \cdots, n_m}  ( \chi_E) 
$$
$$
\ge \chi_{E_1} \left(1-  2  \exp\Bl(-\f {cm} {Kv \log^2 (2Kv)}\Br)\right),
$$
where $\EE_{n_1, \cdots, n_m}  ( \chi_E) = \EE[ \chi_E |x_1, \cdots, x_{m_1}]$ denotes the conditional expectation given $x_1,\cdots, x_{m_1}$; that is, the expectation computed with respect to the random variables $n_1, \cdots, n_m$,  fixing the variables  $x_1,\cdots, x_{m_1}$.  
Since $\chi_{E_1}$  is a function of $(x_1, x_2, \cdots, x_{m_1})$, it follows that 
\begin{align*}
\PP (E)& \ge \PP (E\cap E_1) =\EE[ \chi_{E\cap E_1}]=
\EE_{x_1, \cdots, x_n}\Bl[\chi_{E_1} \cdot \EE_{n_1, \cdots, n_m} ( \chi_{E})\Br]\\
&\ge \Bl(  1-  2  \exp\Bl(-\f {cm} {K v \log^2 (2Kv)}\Br)\Br) \cdot \EE [ \chi_{E_1}]\\
&\ge\Bl(  1-  2  \exp\Bl(-\f {cm} {Kv \log^2 (2Kv)}\Br)\Br)\Bl( 1-4 e^{-c mv}\Br)\\
&\ge   1-  2   \exp\Bl(-\f {c'm} {Kv \log^2 (2Kv)}\Br).
\end{align*}
This completes the proof.

	\section{A lower bound for universal discretization}
	\label{lb}
	
	The goal in this section is to construct  an example of a 
	uniformly bounded Riesz system   $\CD=\{\vi_j\}_{j=1}^\infty$  such  that for $\CD_N=\{\vi_j\}_{j=1}^N$ with $N=1,2,\cdots$ , 
	\begin{equation}\label{4-1}
	m(\mathbf{\mathcal{X}}_v(\D_N))\ge c \log N,  \  \   v=1,2,\cdots, N, 
	\end{equation}
	where     $c>0$ is an absolute constant. 
	By monotonicity, it is enough to show this last inequality for $v=1$.

	Let $L_2[0, 1]$ denote the Lebesgue $L_2$-space defined with respect to usual Lebesgue measure on $[0, 1]$. 
	Let
	$$\D:=\Bl\{ \f 1 {\sqrt{2}} \sin (\pi k x):\  \ k=1,2,\cdots\Br\}$$ and
	$$ \D_N :=\Bl\{ \f 1 {\sqrt{2}} \sin (\pi k x):\  \ k=1,2,\cdots, N\Br\},\  \ N=1,2,\cdots.$$
 Then $\CD$ is a  uniformly bounded orthonormal system    in $L_2[0, 1]$.
  Assume that $\{\xi^\nu\}_{\nu=1}^m$ is a set of $m$ points in $[0, 1]$ that provides universal sampling discretization for the collection 
  $\mathbf {\mathcal{X}}_1(\CD_N)$. Thus, 
  \begin{equation}\label{4-2}
  \f 12 \|g\|_2^2 \leq \f 1m \sum_{\nu=1}^m |g(\xi^\nu) |^2,\   \ \forall g\in\CD_N.
  \end{equation}
	We claim that \eqref{4-2} implies \eqref{4-1}.
	
	To show this claim, we need the simultaneous version of the Dirichlet's Theorem on diophantine approximation: 
	\begin{Lemma}\label{lem-4-1}\textnormal{\cite[p. 27, Theorem 1A]{S}}  
			Given any $\xi^1, \cdots, \xi^m \in [0, 1]$ and any  positive integer $N$,  there exist integers  $1\leq k \le N$ and
	$a_1,\dots,a_m\in\ZZ$ such that 
	\be\label{U2}
	\Bl|\xi^\nu -\f {a_\nu} k \Br| \le  k^{-1}N^{-1/m},\quad \nu=1,\dots,m.
	\ee
\end{Lemma}

Now let $1\leq k\leq N$ and $a_1,\cdots, a_m$ are integers as given in Lemma \ref{lem-4-1}.   Let $g_k(x)=\sin (\pi kx)\in\D_N$. Then on one hand, 
$\|g_k\|_2^2=2$, but on the other hand,  for each $1\leq \nu\leq m$,
\begin{align*}
 |g_k(\xi^\nu)|&=| \sin (\pi k \xi^\nu)|=| \sin (\pi k \xi^\nu)-\sin (\pi  a_\nu)|\leq | \pi k \xi^\nu-\pi a_\nu|\\
 &=\pi k \Bl| \xi^\nu-\f {a_\nu} k\Br|\leq \pi N^{-1/m}.
\end{align*}
Thus, using \eqref{4-2}, we obtain 
\[ 1 =\f 12 \|g_k\|_2^2\leq \f 1m \sum_{\nu=1}^m |g_k(\xi^\nu)|^2 \leq \pi N^{-1/m},\]
implying 
\[ m\ge \f {\log N} {\log \pi}.\]

\section{Sparse sampling recovery}
\label{ssr}

We begin this section with the observation that the universal discretization of the $L_2$ norm is closely connected with the concept of Restricted Isometry Property (RIP), which is very important in
compressed sensing. 

 Let $U=\{\bu^j\}_{j=1}^N$ be a system of column vectors   from
 $\bbC^m$. We form a matrix $\bU:= [\bu^1\  \bu^2\  ...\  \bu^N]$ with vectors $\bu^j$ being the columns of this matrix. We say that matrix $\bU$  or the system $U$ has the RIP property with parameters $D$ and $\delta \in (0,1)$ if for any subset $J$ of indexes from $\{1,2,\dots,N\}$ with cardinality $|J|\le D$,  and for any vector $\ba=(a_1,\dots,a_N)\in \bbC^N$  supported on $J$ (i.e., $a_j=0, j\notin J$), we have
 $$
 (1-\delta)\|\ba\|_2^2 \le \left\|\sum_{j\in J} a_j\bu^j\right\|_2^2=\|\bU\ba\|_2^2 \le  (1+\delta)\|\ba\|_2^2.
 $$
 Let $\Omega$ be a compact subset of $\R^d$ with the probability measure $\mu$. 
  For a dictionary $\D_N=\{\vi_i\}_{i=1}^N$ and a set of points $\xi:= \{\xi^j\}_{j=1}^m \subset \Omega $ consider the system of vectors $G_N(\xi)=\{\pmb \vi_i(\xi) \}_{i=1}^N$ where 
  $$
  \pmb\vi_i(\xi) :=m^{-\f12} S(\vi_i,\xi)=m^{-1/2} (\vi_i(\xi^1),\dots,\vi_i(\xi^m))^T,\quad  i=1,\dots,N.
  $$ 
Suppose that $\D_N$ is an orthonormal system.   On one hand,
for any $v$-sparse vector  $\ba=(a_1,\cdots, a_N)\in\CC^N$ that is supported on a set $I\subset\{1,2,\cdots, N\}$  of cardinality $v$, \  \  \footnote{This means that  $a_i=0$ for $i\notin I$} and for
  $f= \sum_{i\in I}a_i\vi_i\in \Sigma_v(\D_N)$, we have
$$
\|f\|_2^2 = \|\ba\|_2^2, \quad S(f,\xi)= (f(\xi^1),\dots,f(\xi^m))^T = m^{1/2}\sum_{i\in I} a_i\pmb\vi_i.
$$
On the other hand, however,  
$$
\left\|\sum_{i\in I} a_i\pmb  \vi_i\right\|_2^2 = \frac{1}{m}\sum_{j=1}^m |f(\xi^j)|^2.
$$
Thus,  the set $\xi:= \{\xi^j\}_{j=1}^m \subset \Omega $ provides {\it universal discretization} of the $L_2$ norm for the collection $\cX_v(\D_N)$ with constants $C_1=1-\delta$ and $C_2=1+\delta$ if and only if the  the system $G_N(\xi)$ has the RIP property with parameters $v$ and $\delta \in (0,1)$.
 
 The reader can find results on the RIP properties of systems $G_N(\xi)$ associated with uniformly bounded orthonormal systems $\D_N$ in the book \cite{FR}. For illustration we formulate the following  result 
 from \cite{FR}. 
 
 \begin{Theorem}\label{FR} \textnormal{\cite[p. 405, Theorem 12.31]{FR} } There exists a universal constant $C$ with the following properties. Let $\D_N$ be uniformly bounded orthonormal system such that $\|g_i\|_\infty \le K$, $i=1,2,\dots,N$. Assume that $\xi^1$,...,$\xi^m$ are independent random points that are identically  distributed  according to $\mu$ on 
 	$\Omega$. Then with probability at least $1-N^{-(\ln N)^3}$ the system $G_N(\xi)$ has the RIP property with parameters $s$ and $\delta$ provided
 \be\label{FR1}
 m\ge CK^{-2}\delta^{-2}s(\ln N)^4.
 \ee
 \end{Theorem}

 We now turn to the proof of Theorem \ref{IT2}.
 It   relies on  a result from \cite{LT} (see also \cite{VTbookMA}, Section 8.7) under the following assumption on the dictionary.\\

 {\bf UP($u,D$). ($u,D$)-unconditional property.}   We say that a dictionary $\D=\{\vi_i\}_{i\in I}$ of  elements 
 in a Hilbert space $H=(H, \|\cdot\|)$ 
  is ($u,D$)-unconditional with  constant $U>0$ for some integers $1\leq u\leq D$ if for any 
   $f=\sum_{i\in T} c_i \vi_i\in \Sigma_u(\D)$ with  $T\subset I$ and $|T|=u$,  and  for  any $A\subset T$ and  $J\subset I\setminus A$ such that   $|A|+|J| \le D$,  we have
\be\label{UP}
\Bl\|\sum_{i\in A} c_i \vi_i\Br\|\leq U\inf_{g\in V_J}\Bl \|\sum_{i\in A} c_i \vi_i-g\Br\|,
\ee
where $ V_J(\CD):=\spn\{\vi_i:\  \ i\in J\}$.

Recall that the notation  $\text{WOMP}\bigl(\D; t\bigr)_H$     denotes the WOMP  that is  defined with respect to  a  weakness parameter $t\in (0, 1]$ and   a dictionary $\CD$ in a Hilbert  space $H$.

 \begin{Theorem}[{\cite[Corollary I.1]{LT}}]\label{ssrT1} Let $\CD$ be a dictionary in a Hilbert space $H=(H, \|\cdot\|)$ having  the property  {\bf UP($u,D$)}  with   constant $U>0$ for some  integers $1\leq u\leq D$.   Let $f_0\in H$, and let $t\in (0, 1]$ be a given weakness parameter. 
  Then there exists a positive constant $c_\ast:=c(t,U)$  depending only on $t$ and  $U$ such that  the $\text{WOMP}\bigl(\D; t\bigr)_H$ applied to  $f_0$ gives
  $$
  \left\|f_{\left \lceil{c_\ast v}\right\rceil} \right\| \le C\sigma_v(f_0,\D)_H,\  \ v=1,2,\cdots, \min\Bl\{u,   \floor*{ D/{(1+c_\ast)}}\Br\},
  $$
   where  $C>1$ is an absolute constant, and $f_k$ denotes the residue of $f_0$ after the $k$-th iteration of the algorithm.

\end{Theorem}

   
\begin{proof}[Proof of Theorem \ref{IT2}]
Using (\ref{Riesz0}), we obtain that  for any  sets $A\subset \{1,2,\cdots, N\}$ and  $\Lambda\subset \{1,\cdots, N\}\setminus A$ , and for  any  sequences $\{x_i\}_{i\in A},$ $\{c_i\}_{i\in\Ld}\subset\CC$, 
$$
\Bl\|\sum_{i\in A}x_i\ff_i-\sum_{i\in\Lambda}c_i\ff_i\Br\|_{L_2(\Omega,\mu)}^2 \ge R_1^2  \sum_{i\in A} |x_i|^2 \ge R_2^{-2} R_1^2\Bl \|\sum_{i\in A}x_i\ff_i \Br\|_{L_2(\Omega,\mu)}^2,
$$
meaning  that the dictionary  $\CD_N$ has the {\bf UP}$(v,N)$ property  with constant $U_1 = R_2/R_1$  in the  space $L_2(\Omega,\mu)$  for any integer $1\leq v< N$. 
It then follows from the  discretization inequalities  
(\ref{I3}) with $\mathbf {\mathcal{X}}=\mathbf {\mathcal{X}}_u(\D_N)$ that the dictionary $\D:=\CD_N(\Omega_m)$, which is   the restriction 
of $\CD_N$ on $\Omega_m=\{\xi^1, \cdots,\xi^m\}$, has  
the  properties  {\bf UP($v,u$)} in the space  $L_2(\Omega_m,\mu_m)$  with  constant $U_2 =U_1 3^{1/2}$ for any integer $1\leq v\leq u$.
Thus,  applying Theorem \ref{ssrT1} to the discretized dictionary  $\CD_N(\Omega_m)$ in the Hilbert space  $L_2(\Omega_m,\mu_m)$, we conclude that   the algorithm  $$\text{WOMP}\bigl(\CD_N(\Omega_m); t\bigr)_{L_2(\Omega_m,\mu_m)}$$ applied to  $f_0\Bl|_{\Og_m} $ gives
$$
\left\|f_{cv} \right\|_{L_2(\Omega_m,\mu_m)}\le C\sigma_v(f_0,\CD_N(\Omega_m))_{L_2(\Omega_m,\mu_m)}
$$
whenever $v+c v\leq u$, where  $c =c(t, R_1, R_2)\in\NN$. 
 This proves (\ref{mp}) in Theorem \ref{IT2}. 

We now derive (\ref{mp2}) from (\ref{mp}).   Clearly, 
$$
\sigma_v(f_0,\CD_N(\Omega_m))_{L_2(\Omega_m,\mu_m)} \le \sigma_v(f_0,\CD_N )_\infty.
$$
Let $f\in \Sigma_v(\CD_N)$ be such that  $\|f_0-f\|_\infty \le 2 \sigma_v(f_0,\CD_N)_\infty$. Then (\ref{mp}) implies 
$$
\|f - G_{cv}(f_0,\CD_N(\Omega_m))\|_{L_2(\Omega_m,\mu_m)} \le \|f-f_0\|_{L_2(\Omega_m,\mu_m)} +\|f_{cv}\|_{L_2(\Omega_m,\mu_m)} 
$$
$$
\le (2+C)\sigma_v(f_0,\CD_N)_\infty.
$$
Using that $f - G_{cv}(f_0,\CD_N(\Omega_m)) \in \Sigma_u(\CD_N)$, by discretization (\ref{I3}) we 
conclude that
\be\label{ssr3}
\|f - G_{cv}(f_0,\CD_N(\Omega_m))\|_{L_2(\Omega,\mu)} \le 2^{1/2}(2+C)\sigma_v(f_0,\CD_N)_\infty.
\ee
Finally,
$$
\|f_{cv}\|_{L_2(\Omega,\mu)} \le \|f-f_0\|_{L_2(\Omega,\mu)} + \|f - G_{cv}(f_0,\CD_N(\Omega_m))\|_{L_2(\Omega,\mu)}.
$$
This and (\ref{ssr3}) prove (\ref{mp2}).

\end{proof}


Universal discretization also has  some  interesting applications in sparse sampling 
recovery in the norm $L_2$, as was shown in \cite{DT3}.  To see this, we assume that $\Og$ is a compact domain for the sake of  convenience. 
Given a finite-dimensional subspace  $X$ of  bounded  functions on $\Og$, and  a vector  $ \xi=(\xi^1,\cdots,\xi^m)\in  \Omega^m$,  the 
 classical  least squares recovery operator (algorithm) is defined as  (see, for instance, \cite{CM})
 $$
 LS(\xi,X)(f):=\text{arg}\min_{u\in X} \|S(f-u,\xi)\|_{2},\quad
 $$
 where $S(g,\xi):=(g(\xi^1), \cdots,  g(\xi^m))$, and 
 $$
 \|S(g,\xi)\|_{2}:= \left(\f 1m \sum_{\nu=1}^m  |g(\xi^\nu)|^2\right)^{1/2} .
 $$
 With the help of the classical least squares  algorithms, 
 we  define in  \cite{DT3} a new  nonlinear  algorithm  for a given collection $\cX=\{X(n)\}_{n=1}^k$ of finite-dimensional subspaces of $\cC(\Og)$:
 $$
n(\xi,f) := \text{arg}\min_{ 1\le n\le k}\|f-LS(\xi,X(n))(f)\|_2,
$$
\be\label{I4}
  LS(\xi,\cX)(f):= LS(\xi,X(n(\xi,f)))(f).
\ee

 \begin{Definition}\label{ID1} We say that a set $\xi:= \{\xi^j\}_{j=1}^m \subset \Omega $ provides {\it one-sided universal discretization} with constant $C_1$ for a  collection $\cX:= \{X(n)\}_{n=1}^k$ of finite-dimensional  linear subspaces of functions on $\Og$ if 
 $$
C_1\|f\|_2^2 \le \frac{1}{m} \sum_{j=1}^m |f(\xi^j)|^2\quad \text{for any}\quad f\in \bigcup_{n=1}^k X(n) .
$$
We denote by $m(\cX,C_1)$ the minimal positive integer  $m$ such that there exists a set $\xi$ of $m$ points, which
provides one-sided universal discretization with constant $C_1$ for the collection $\cX$. 
\end{Definition}

For $\xi=(\xi^1,\cdots,\xi^m)\in \Omega^m$ let $\mu_\xi$ denote the probability measure
\begin{equation}\label{5-5}
\mu_\xi := \frac{1}{2} \mu + \frac{1}{2m} \sum_{j=1}^m \delta_{\xi^j},
\end{equation}
where $\delta_\bx$ denotes the Dirac measure supported at a point $\bx$.

We proved the following  conditional theorem in \cite{DT3}.

 \begin{Theorem}[\cite{DT3}, Theorem 1.2]\label{ssrT2} Let $v$, $N$ be given natural numbers such that $v\le N$.  Let $\D_N\subset \C(\Og)$ be  a dictionary of $N$ continuous functions on $\Og$. Assume that  there exists a set $\xi:= \{\xi^j\}_{j=1}^m \subset \Omega $, which provides {\it one-sided universal discretization} with constant $C_1$ for the collection $\cX_v(\D_N)$. Then for   any  function $ f \in \C(\Omega)$ we have
\be\label{I5}
  \|f-LS(\xi,\cX_v(\D_N))(f)\|_2 \le 2^{1/2}(2C_1^{-1} +1) \sigma_v(f,\D_N)_{L_2(\Og, \mu_\xi)}
 \ee
 and
 \be\label{I6}
  \|f-LS(\xi,\cX_v(\D_N))(f)\|_2 \le  (2C_1^{-1} +1) \sigma_v(f,\D_N)_\infty,
 \ee
 where $\mu_\xi$ is the probability measure given in \eqref{5-5}.
 \end{Theorem}
 
 \begin{Remark}
 	 An advantage of the algorithm $LS(\xi,\cX_v(\D_N))$ over the WOMP is that it provides in (\ref{I5}) and (\ref{I6}) the error in the $L_2(\Omega,\mu)$ norm while  the WOMP 
 provides in (\ref{mp})  the error in the discrete $L_2(\Omega_m,\mu_m)$ norm. However, a big advantage of the WOMP over $LS(\xi,\cX_v(\D_N))$ is that it is a simple algorithm, which is known 
 for easy practical implementation, whereas  the $LS(\xi,\cX_v(\D_N))$ is only a theoretical algorithm, which has the step (\ref{I4}) that may be difficult to realize. Both the algorithms WOMP and $LS(\xi,\cX_v(\D_N))$ only use information from $f$ and provide the error close to the best (in a certain sense). They do not use the information that $f\in \bF$ and automatically provide the error bound in terms of the class $\bF$, to which $f$ belongs. This makes these algorithms {\it universal}. 
 
\end{Remark}

 In  Theorem \ref{IT2},  the  WOMP 
provides an  error in the discrete norm $L_2(\Omega_m,\mu_m)$ in the estimate (\ref{mp}). However,
a slight modification of the above  proof of Theorem \ref{IT2} also yields the following corollary, where  the error  is measured in the $L_2(\Omega,\mu)$ norm
rather than the the discrete norm $L_2(\Omega_m,\mu_m)$.

 \begin{Corollary}\label{ssrR1}
 	 Under the conditions of Theorem \ref{IT2}, we have 
 \be\label{mp3}
  \|f_{cv} \|_{L_2(\Omega,\mu)} \le C \sigma_v(f_0,\CD_N)_{L_2(\Og, \mu_\xi)},
 \ee
 where   $c=c(t,R_1,R_2)\ge 1$ is  the  constant integer given in Theorem \ref{IT2}, and  $f_k$ is
  the residue of $f_0$ after the $k$-th iteration of the algorithm $$\text{WOMP}\Bl(\D_N(\Og_m); t\Br)_{L_2(\Omega_m,\mu_m)}.$$
  \end{Corollary}
 \begin{proof}
 	Recall that $$\Og_m=\{\xi^1,\cdots, \xi^m\},\   \  \mu_m= \frac1{m}\sum_{j=1}^m \delta_{\xi^j}\  \  \ \text{and}\  \ \mu_\xi = \f{\mu+\mu_m}2.$$
 	For convenience, we will use the notation $\|\cdot\|_{L_2(\nu)}$  to denote the norm of $L_2$ defined with respect to a measure $\nu$ on $\Og$. 
 	 	Let $g\in \Sigma_v(\CD_N)$ be such that  $\|f_0-g\|_{L_2(\mu_\xi)} \le 2 \sigma_v(f_0,\CD_N)_{L_2(\mu_\xi)}$.
 	Then \begin{align*}
 	  \|f_{cv} \|_{L_2(\mu)}&\leq
 	   \sqrt{2} \|f_0 - G_{cv}(f_0,\CD_N(\Omega_m))\|_{L_2(\mu_\xi)}\\
 	   &\leq \sqrt{2}\|f_0-g\|_{L_2(\mu_\xi)}+\sqrt{2} \|g- G_{cv}(f_0,\CD_N(\Omega_m))\|_{L_2(\mu_\xi)}\\
 	   &\leq 2\sqrt{2} \sigma_v(f_0,\CD_N)_{L_2(\mu_\xi)}+ \sqrt{2} \|g- G_{cv}(f_0,\CD_N(\Omega_m))\|_{L_2(\mu_\xi)}.
 	\end{align*}
 	Since
 	\[ g- G_{cv}(f_0,\CD_N(\Omega_m))\in \Sigma_{v+cv} (\CD_N) \subset \Sigma_u(\CD_N),\]
 	it follows by the universal discretization that 
 	\begin{align*} \|g-& G_{cv}(f_0,\CD_N(\Omega_m))\|_{L_2(\mu_\xi)}\leq C  \|g- G_{cv}(f_0,\CD_N(\Omega_m))\|_{L_2(\mu_m)}\\
 &	\leq C \|f_0-g\|_{L_2(\mu_m)}+C  \|f_0- G_{cv}(f_0,\CD_N(\Omega_m))\|_{L_2(\mu_m)}\\
 &\leq C \|f_0-g\|_{L_2(\mu_\xi)}+ C \|f_{cv} \|_{L_2(\mu_m)},\end{align*}
 	which, using Theorem \ref{IT2}, is estimated by 
 	\begin{align*}
 	&\leq C  \sigma_v(f_0,\CD_N)_{L_2(\mu_\xi)}+ C\sigma_v(f_0,\CD_N(\Omega_m))_{L_2(\mu_m)}\leq C  \sigma_v(f_0,\CD_N)_{L_2(\mu_\xi)}.
 	\end{align*}

\end{proof}

 Finally,  we discuss applications of the  universal discretization for the trigonometric system in  sampling recovery of periodic functions. 
 Let $M\in \N$ and $d\in \N$. Let $\Pi(M) := [-M,M]^d$ denote the $d$-dimensional cube. Consider the system $$\Tr(M,d) :=\Bl\{ e^{i(\bk,\bx)}: 
 \bk \in \Pi(M)\cap\ZZ^d\Br\}$$
 of trigonometric  function on $\T^d =[0,2\pi)^d$. Then $\Tr(M,d)$ is an orthonormal system in $L_2(\T^d,\mu)$ with $\mu$ being the normalized Lebesgue measure on $\T^d$. The cardinality of this system is $N(M):= |\Tr(M,d)| = (2M+1)^d$. In our further applications we are interested in bounds on $m(\cX_v(\Tr(M,d))$ in a special case when $M\le v^c$ with some constant $c$ which may depend on $d$. Here we recall that  the notation  $m(\cX)$  is defined in the introduction  as    the minimal number of points  required for  the universal discretization (\ref{I3})  for a  collection $\cX$ of finite-dimensional  linear subspaces.
 Theorem \ref{IT1} 
 gives 
 \be\label{DT1}
 m(\cX_v(\Tr(v^c,d)) \le C(c,d) v (\log (2v))^4.
 \ee
 This  can also be deuced from (\ref{FR1}) in Theorem \ref{FR} on the 
  RIP properties.  However, it is known that in the case of system $\Tr(M,d)$ the bound (\ref{FR1}) in Theorem \ref{FR} can be improved, which in turn can be used  to improve the bound (\ref{DT1}). To be more precise, combining   results of \cite{HR} and \cite{Bour}, and using the  argument in the beginning of this section, we get 
 \be\label{HR1}
 m(\cX_v(\Tr(v^c,d)) \le C(c,d) v (\log (2v))^3.
 \ee
 
 We now apply  Corollary \ref{ssrR1} to optimal sampling recovery of periodic functions. This discussion complements the one from \cite[Section 5]{DT3}.
 Given a positive integer $N$, let 
 $$
 \Gamma(N) := \bigl\{ \mathbf k\in\mathbb Z^d :\prod_{j=1}^d
 \max\bigl( |k_j|,1\bigr) \le N\bigr\}\quad\text{--}\quad\text{a hyperbolic cross}.
 $$
 Given a finite subset  $Q\subset \mathbb Z^d$, define 
 $$
 \Tr(Q) :=\left\{ t : t(\mathbf x) =\sum_{\mathbf k\in Q}c_{\mathbf k}
 e^{i(\mathbf k,\mathbf x)},\  \ c_{\mathbf k} \in\CC \right\} .
 $$

  For a vector  $\mathbf s=(s_1,\dots,s_d )$  whose  coordinates  are
nonnegative integers, we define
$$
\rho(\mathbf s) := \bigl\{ \mathbf k\in\mathbb Z^d:  \floor*{ 2^{s_j-1}} \le
|k_j| < 2^{s_j},\qquad j=1,\dots,d \bigr\},
$$
and define, for $f\in L_1 (\T^d)$, 
$$
\delta_{\mathbf s} (f,\mathbf x) :=\sum_{\mathbf k\in\rho(\mathbf s)}
\hat f(\mathbf k)e^{i(\mathbf k,\mathbf x)},\quad \hat f(\mathbf k) := (2\pi)^{-d}\int_{\T^d} f(\bx)e^{-i(\mathbf k,\mathbf x)}d\bx.
$$
We also define  for $f\in L_1(\T^d)$
$$
f_j:=\sum_{\|\bs\|_1=j}\delta_\bs(f), \quad j=0,1,\cdots.
$$
The  Wiener norm (the $A$-norm or the $\ell_1$-norm) of $f\in L_1(\T^d)$ is defined as 
$$
\|f\|_A := \sum_{\mathbf k\in\ZZ^d}|\hat f({\mathbf k})|.
$$
The following classes, which are convenient in studying sparse approximation with respect to the trigonometric system, 
were  introduced and studied in \cite{VT150}. 
For parameters $ a\in \R_+$ and $ b\in \R$,  define 
$$
\bW^{a,b}_A:=\Bl\{f: \|f_j\|_A \le 2^{-aj}(j+1)^{(d-1)b},\quad  j=0,1,\cdots\Br\}.
$$
The following result was presented in  \cite{DT3} without proof. 

\begin{Lemma}\label{ssrL1}   There exist two constants $c(a,d)$ and $C(a,b,d)$ such that for any $\xi\in (\T^d)^m$ and $v\in\N$ there is a constructive method $A_{v,\xi}$ based on greedy algorithms, which provides a $v$-term approximant 
from $\Tr(\Gamma(M))$, $|\Gamma(M)|\le v^{c(a,d)}$, with 
the bound for $f\in \bW^{a,b}_A$
$$
\|f-A_{v,\xi}(f)\|_{L_2(\T^d,\mu_\xi)} \le C(a,b,d)  v^{-a-1/2} (\log v)^{(d-1)(a+b)}.      
$$
\end{Lemma}

 For completeness we 
will give   a proof of a somewhat more general statement at the end of this section.

    Let $v\in\N$ be given and let $M$ be from Lemma \ref{ssrL1}. Consider the orthonormal 
basis $\CD_N:=\{e^{i(\bk,\bx)}\}_{\bk\in \Gamma(M)}$ of the space $\Tr(\Gamma(M))$.
Then Lemma \ref{ssrL1},  Corollary  \ref{ssrR1} and the bound (\ref{HR1}) imply the following   Theorem \ref{ssrT3}.

\begin{Theorem}\label{ssrT3}   There exist two constants $c'(a,d)$ and $C'(a,b,d)$ such that       for any $v\in\N$ we have that  for $a>0$, $b\in \R$,
\begin{equation}\label{ssr10}
 \varrho_{m}^{o}(\bW^{a,b}_A,L_2(\T^d)) \le C'(a,b,d)  v^{-a-1/2} (\log v)^{(d-1)(a+b)}      
\end{equation}
provided that $m$ is an integer satisfying 
$$
m\ge c'(a,d) v(\log(2v))^3.
$$
\end{Theorem}

Theorem \ref{ssrT3} gives the following bound for the classical classes $\bW^r_p$ of functions with bounded mixed derivative (see \cite{VTbookMA}, p.130, for their definition and \cite{DT3} for some recovery results). 

\begin{Corollary}\label{ssrT4}   There exist two constants $c'(r,d,p)$ and $C'(r,d,p)$ such that       for any $v\in\N$ we have that for $r>1/p$ and  $1<p< 2$,
\begin{equation}\label{ssr11}
 \varrho_{m}^{o}(\bW^{r}_p,L_2(\T^d)) \le C'(r,d,p)  v^{-r+1/p-1/2} (\log v)^{(d-1)(r+1-2/p)}      
\end{equation}
provided that
$$
m\ge c'(r,d,p) v(\log(2v))^3.
$$
\end{Corollary}

The authors of \cite{JUV} (see Corollary 4.16 in v2) proved the following interesting bound for $1<p<2$, $r>1/p$ and  $m\ge c(r,d,p) v(\log(2v))^3$,
\begin{equation}\label{D2}
\varrho_{m}^{o}(\bW^{r}_p,L_2(\T^d)) \le C(r,d,p)  v^{-r+1/p-1/2} (\log v)^{(d-1)(r+1-2/p)+1/2}      
\end{equation}
The bound (\ref{ssr11}) is better than (\ref{D2}) by the $(\log v)^{-1/2}$ factor. The proof of (\ref{D2}) in \cite{JUV} is based on known results from \cite{VT150} on sparse trigonometric approximation in the uniform norm. We proved (\ref{ssr11}) by using inequality (\ref{mp3}) and estimating the sparse trigonometric approximation in the norm $L_2(\T^d,\mu_\xi)$, which is weaker than the uniform norm.

We conclude this section with a proof of Lemma \ref{ssrL1}. Indeed, we shall prove a somewhat more general result, where the trigonometric system is replaced by a more general system $\D=\{g_\bk\}_{\bk\in \Z^d}$  uniformly bounded functions on a domain $\Og$:
\be\label{ssr14}
|g_\bk(\bx)| \le 1,\quad \bx \in \Omega, \quad \bk \in \Z^d.
\ee
As in the trigonometric case, for 
\be\label{ssr12}
f=\sum_{\bk\in \ZZ^d}a_\bk g_\bk\   \   \text{ with }\  \ \sum_{\bk\in \ZZ^d}|a_{\bk}|<\infty,
\ee
 we define 
$$
\delta_\bs(f,\D):= \sum_{\bk\in \rho(\bs)}a_\bk g_\bk,\  \   f_j:=\sum_{\|\bs\|_1=j}\delta_\bs(f,\D), j=0,1,\cdots.
$$
and
$$
\|f\|_{A(\D)} := \sum_{\bk\in\ZZ^d}|a_\bk|.
$$
For parameters $ a\in \R_+$, $ b\in \R$ define  $\bW^{a,b}_A(\D)$  to be the class  of all functions $f$ with  a representation (\ref{ssr12}) satisfying
\be\label{ssr15}
 \|f_j\|_{A(\D)} \le 2^{-aj}(\bar j)^{(d-1)b},\quad j=0,1,\cdots,
\ee
where $\bar{j}=\max\{j,1\}$.
Given  a finite set $Q$ of points in $\mathbb Z^d$, we denote
$$
\D(Q) :=\left\{ \sum_{\mathbf k\in Q}a_{\mathbf k}
g_\bk:\  \  a_{\bk}\in\CC,\  \ \bk\in Q\right\} .
$$

\begin{Lemma}\label{ssrL2}   There exist two constants $c(a,d)$ and $C(a,b,d)$ such that for any $\xi\in \Omega^m$ and $v\in\N$ there is a constructive method $A_{v,\xi}$ based on greedy algorithms, which provides a $v$-term approximant 
from $\D(\Gamma(M))$, $|\Gamma(M)|\le v^{c(a,d)}$, with 
the bound for $f\in \bW^{a,b}_A(\D)$
$$
\|f-A_{v,\xi}(f)\|_{L_2(\Omega,\mu_\xi)} \le C(a,b,d)  v^{-a-1/2} (\log v)^{(d-1)(a+b)}.      
$$
\end{Lemma}
\begin{proof}  We prove the lemma for $v\asymp 2^nn^{d-1}$, $n\in \N$. Let 
$$
f=\sum_{\bk}a_\bk g_\bk
$$
be a representation of $f\in \bW^{a,b}_A(\D)$ satisfying (\ref{ssr15}). We approximate $f_j$ in $L_2(\Omega,\mu_\xi)$ by 
a $v_j$-term approximant from  $\D(Q_j) := \cup_{\|\bs\|_1=j} \rho(\bs)$. The sequence $\{v_j\}$ will be defined later.
Since   $L_2(\Og, \mu_\xi)$ is a Hilbert space and by (\ref{ssr14}) 
 	\be\label{ub}
	\|g_\bk\|_{L_2(\Og, \mu_\xi)}\leq \sup_{\bx\in\Og} |g_\bk(\bx)|\leq 1,\   \bk \in Q_j,
	\ee
 	using Theorem 2.19 of \cite[p. 93]{VTbook}, we deduce from \eqref{ssr15}
 	 that for every $j$  there exists an $h_j\in\Sigma_{v_j}(\D(Q_j))$ (provided by a greedy algorithm -- Orthogonal Matching Pursuit)
 	such that
 $$
 \|f_j-h_j\|_{L_2(\Og, \mu_\xi)}\leq  v_j^{-1/2}2^{-aj}(\bar j)^{(d-1)b}.
 $$
 We take $\beta\in (0,a)$ (for instance, $\beta=a/2$) and specify
$$
v_j := [2^{n-\beta (j-n)}j^{d-1}],\quad j=n,n+1,\dots.
$$
It is clear that there exists $j(\beta,d)$ such that $v_j=0$ for $j\ge j(\beta,d)n$. For $j\ge j(\beta,d)n$
we set $h_j=0$.
In addition to $\{h_j\}_{j=n}^{j(\beta,d)n}$ we include in the approximant 
$$
S_n(f,\D) := \sum_{\|\bs\|_1< n}\delta_\bs(f,\D).
$$
Define
$$
A_v(f,\beta) := S_n(f,\D)+\sum_{j \ge n} h_j = S_n(f,\D)+\sum_{j = n}^{j(\beta,d)n} h_j .
$$
Then, we have built a $v$-term approximant of $f$ with 
$$
v\ll 2^nn^{d-1}  +\sum_{j\ge n} v_j \ll 2^nn^{d-1}.
$$
 The error   of this approximation in $L_2(\Omega,\mu_\xi)$ is bounded from above by
$$ 
\|f-A_m(f,p,\beta)\|_{L_2(\Omega,\mu_\xi)} \le \sum_{j\ge n} \|f_j-h_j\|_{L_2(\Omega,\mu_\xi)} \leq C \sum_{j\ge n} (\bar v_j)^{-1/2}2^{-aj}j^{(d-1)b}
$$
$$
\leq C \sum_{j\ge n}2^{-1/2(n-\beta(j-n))}j^{-(d-1)/2}2^{-aj}j^{(d-1)b} \leq C 2^{-n(a+1/2)}n^{(d-1)(b-1/2)}.
$$
This completes the proof of lemma.

\end{proof}



\section{A more general  version of the conditional theorem  }\label{sec:5}

%

Our goal in this section  is to prove the following conditional theorem, which is the random version of Theorem 5.1 of \cite{DT}.
\begin{Theorem}\label{thm-4-2} Let  $\CW\subset L_\infty(\Og)$ and let $\CW_p:=\{ f\in\CW:\  \ \|f\|_p\leq 1\}$ for some $1\leq p<\infty$. Assume that $\ld\cdot \CW=\CW$ for any $\ld>0$, and $\CW_p$ is a compact subset of $L_\infty$. Let 
	\[ R:=\sup_{f\in \CW_p} 	\sup_{x\in\Og}  |f(x)|.\]
	Let $\xi^1, \cdots, \xi^m$ be independent random points  on $\Og$ satisfying \begin{equation}\label{6-1}
	\f 1m \sum_{k=1}^m \mu_{\xi^k}=\mu,
	\end{equation}
	where $\mu_{\xi^j}$ denotes the probability distribution of $\xi^j$. 
	Then there exist positive constants $C_p, c_p$ depending only on $p$   such that for any  $\va\in (0, 1)$ and  any 
	integer 
	\begin{equation}\label{5-2-c}
	m\ge  C_p 	  \va^{-5 }\left(\int_{c_p \va^{1/p}} ^{R} u^{\frac  p2-1}    \Bigl(\int_{  u}^{ R }\frac {\cH_{c_p \va t}(\CW_p,L_\infty)}t\, dt\Bigr)^{\frac 12} du\right)^2, 
	\end{equation}
	the inequalities, 
	\begin{equation}\label{5.3b}
	(1-\va) \|f\|_p^p \leq \frac  1m \sum_{j=1}^m |f(\xi^j)|^p\leq (1+\va) \|f\|_p^p,\   \ \forall f\in\CW_p,
	\end{equation} 
	hold  with probability 
	\[\geq 1-  \va   \exp\Bl(-\f {cm\va^4} {R^p (\log \f R\va)^2}\Br).\]
\end{Theorem}

It can be easily seen that either of the following two conditions implies the condition  \eqref{6-1}: 
	\begin{enumerate}[\rm (i)]
		\item  $\xi^1,\cdots,\xi^m$ are identically distributed according to  $\mu$;
		\item there exists a partition $\{\Ld_1, \cdots, \Ld_m\}$ of $\Og$ such that $\mu(\Ld_j)=\f1m$ and $\xi^j\in\Ld_j$ is  distributed  according to    $m\cdot \mu\Bl|_{\Ld_j}$  for each $1\leq j\leq m$. 
	\end{enumerate}

 The proof of Theorem \ref{thm-4-2}  follows along the same line as that of Theorem~ 5.1 of \cite{DT}. We sketch it as follows.

\begin{proof} [Proof of Theorem \ref{thm-4-2}]
	Let 
	\[\CS_p:= \{  f/ {\|f\|_p}:\  \ f\in\CW,\  \ \|f\|_p>0\}\subset \CW_p.\] 
%
%
	Let $c^\ast=c_p^\ast\in (0, \frac 12)$ be a  sufficiently small  constant depending only on $p$.
	Let $a:=c^\ast\va$. Let  $J, j_0$ be two integers  such that $j_0<0\leq J$, 
	\begin{equation}\label{5.8}
	(1+a)^{J-1}\le R < (1+a)^J\   \ \text{and}\   \ (1+a)^{j_0 p} \leq \frac  1{10} \va  \leq (1+a)^{(j_0+1) p }. 
	\end{equation}
	For $j\in\ZZ$, let 
	$$
	\cA_j := \cN_{2a(1+a)^j}(\CS_p,L_\infty)\subset \CS_p
	$$
	denote the minimal $2a(1+a)^j$-net of $\CS_p$ in the norm of $L_\infty$. For $f\in\CS_p$, we  define  $A_j(f)$  to be  the function in  $ {\cA}_j$ such that  $\|A_j (f)-f\|_\infty\leq  2a(1+a)^j$. 
	
	For  $f\in \CS_p$ and $j> j_0$, let
	$
	U_j(f) := \{\bx\in\Og : |A_j(f)(\bx)| \ge (1+a)^{j-1}\},
	$
	and 
	$
	D_j(f) := U_j(f) \setminus \underset{k>j} {\bigcup} U_k(f).$
 Define
	\begin{equation}\label{4-3-0}
		h(f,\bx) := \sum_{j=j_0+1}^{J} (1+a)^j \chi_{D_j(f)}(\bx).
	\end{equation}

It can be easily verified that
	\begin{equation}\label{5-6-eq}
	(1-\f \va8)  |f(\bx)|^p \leq |h(f,\bx)|^p \leq (1+\f \va8) |f(\bx)|^p, 
	\end{equation}
	if $ \bx\in\bigcup_{j_0<j\leq J} D_j(f)$, 
	and   $	|f(\bx)|^p\leq \frac  \va 8$ otherwise \footnote{For details, we refer to the  proof of Theorem 5.1 in  \cite{DT}.} .  As a result, for any probability measure $\nu$ on $\Og$,  and any $f\in\CS_p$, we have 
	\begin{equation}\label{5-9-0}
	\Bl|	\|h(f)\|_{L_p(\nu)}^p- \|f\|_{L_p(\nu)}^p\Br| \leq \f \va 8  \|f\|_{L_p(\nu)}^p+\f \va 8.
	\end{equation}

	%
	%
	
	Now 	for $j_0+1\leq j\leq J$,  let 
	$$
	\cF_j^p := \left\{ M_j\cdot \chi_{D_j(f)}:\   f\in  \CS_p\right\},\   \   \ \text{	where  $M_j:= (1+a)^{pj}$.}
	$$
	Let  $\{\va_j\}_{j=j_0+1}^{J}$ be a sequence of positive numbers satisfying	$\sum_{j=j_0+1}^{J} \va_j \leq \va/4$, which will be  specified later.  
	By \eqref{5-6-eq}, 
	for each  $j_0<j\leq J$, 
	$$
	\|M_j\chi_{D_j(f)}\|_1 \le \|h(f)\|_p^p\leq (1+\f \va 8) \|f\|_p^p\le 2.
	$$
	Thus,  using \eqref{6-1} and   Lemma \ref{lem-2-2}, 
	we  conclude that  the inequalities , 	\begin{align}
	&\sup_{\varphi\in\FF_j^p}\Bigl| \frac  1m \sum_{k=1}^m \varphi(\xi^k) -\int_{\Og} \varphi(x) \, d\mu(x) \Bigr|\leq \va_j,\   \   \    j_0<j\leq  J,\label{5-10a}
	\end{align} 
	hold  simultaneously  with probability 
	\begin{align}
	&\ge 1- \sum_{j=j_0+1}^{J} |\cF_j^p| \exp \Bigl( - \frac  {m\va_j^2} {16M_j} \Bigr).\label{4-9}
	\end{align}
	
	Now we fix  a set of  random points $\xi^1,\cdots, \xi^m$ for which  \eqref{5-10a} is satisfied.   Using  \eqref{4-3-0}, 
	\eqref{5-10a} and  the  triangle inequality, we then deduce 
	\begin{align*}
	\sup_{f\in\CS_p} 	\Bigl| \frac  1 m \sum_{j=1}^m |h(f, \xi^j)|^p-\|h(f)\|_p^p\Bigr|  \leq 	\sum_{j=j_0+1}^{J} \va_j<\f \va4. 
	\end{align*}
	This together with   \eqref{5-9-0} implies  that for  any   $f\in\CS_p$, 
	\begin{align*}
	\Bigl| \frac  1 m \sum_{j=1}^m |f(\xi^j)|^p-1\Bigr|  \leq 	\f 3 8 \va + \f \va 8 \cdot  \frac  1 m \sum_{j=1}^m |f(\xi^j)|^p,
	\end{align*}
	from which the desired inequality \eqref{5.3b} follows.

	Thus, 
	to complete the proof, it remains to construct     a sequence  of positive  numbers $\{\va_j\}_{j=j_0+1}^J$  such that  \begin{equation}\label{4-8}
	\sum_{j=j_0+1}^{J} \va_j \leq \va/4,\end{equation}  and 
	\begin{align}
	&\sum_{j=j_0+1}^{J} |\cF_j^p| \exp \Bigl( - \frac  {m\va_j^2} {16M_j} \Bigr)\leq  C \va \exp\Bl(-\f {cm\va^4} {R^p (\log \f R\va)^2}\Br).\label{4-9}
	\end{align}


	Since  $\CS_p\subset \CW_p$, we can estimate the  cardinalities of the sets $\FF_j^p$  as follows: 
	\begin{equation}\label{5-15}
	|\cF_j^p| \le |\cA_j|\times\cdots\times|\cA_{J}|\leq \prod_{k=j}^J N_{a(1+a)^k} (\CW_p, L_\infty)=:L_j.
	\end{equation}
A straightforward calculation shows that 
	\begin{align}
	\log L_j&= \sum_{k=j}^{J} \cH_{a(1+a)^k}(\CW_p,L_\infty)
	\leq C\va^{-1}  \int_{ ( 1+a)^{j-1}}^{R }\cH_{at}(\CW_p,L_\infty)\frac  {dt}t.\label{4-10}
	\end{align}

	Now we specify the numbers $\va_j$ as follows:
	$$\va_j:=\da+\tau_j,\  \ j_0<j\leq J,$$ where $\da:=\frac {\va} {8 (J-j_0)}$, and $\tau_j>0$ is defined by 
	\begin{align}
	&2\log  L_j = \frac  {m\tau_j^2}{16 M_j},\   \    \  \text{that is}\   \  \tau_j:=4\sqrt{2}\sqrt{\f{M_j\cdot \log  L_j} m} .\label{4-11a}
	\end{align}
	We  have 
	\begin{align*}
	\sum_{j=j_0+1}^{J}\va_j&=4\sqrt{2}m^{-1/2} \sum_{j=j_0+1}^{J} (M_j \log L_j)^{\frac 12} +\f \va 8,
	\end{align*}
	which,  by \eqref{4-10} and straightforward calculations,  is estimated above by 
	\begin{align*}
	&\leq  \f \va 8 + C m^{-\f12} \va^{-\frac 32}\int_{4^{-1}(10^{-1}\va)^{1/p}} ^{R} u^{\frac  p2-1}    \Bigl(\int_{  u}^{ R }\cH_{c_p \va t}(\CW_p,L_\infty)\frac  {dt}t\Bigr)^{\frac 12} du.
	\end{align*}
	Thus,     to ensure \eqref{4-8}, it is enough to assume  that 
	\[ m\ge C \va^{-5}  \Bigg[\int_{4^{-1}(10^{-1}\va)^{1/p}} ^{R} u^{\frac  p2-1}    \Bigl(\int_{  u}^{ R }\cH_{c_p \va t}(\CW_p,L_\infty)\frac  {dt}t\Bigr)^{\frac 12} du\Bigg]^2\]
	with  $C=C_p$ being  a large constant depending only on $p$.

	Finally, we estimate the sum 
	$$S:=\sum_{j=j_0+1}^{J} |\cF_j^p| \exp \Bigl( - \frac  {m\va_j^2} {16M_j} \Bigr). $$
	Since $\va_j^2\ge \tau_j^2+\da^2$, we obtain from \eqref{5-15} that 
	\begin{align*}
	S&\leq 
	\sum_{j=j_0+1}^{J}  \exp \Bigl(\log L_j  - \frac  {m\tau_j^2} {16M_j} - \frac  {m\da^2} {16M_j}\Bigr)
	\leq \exp\Bl(- \frac  {m\da^2} {16M_J}\Br)  \sum_{j=j_0+1}^{J}  \f 1 {L_j}.
	\end{align*}  
Using  \eqref{5.8},  we have 
	\[ \frac  {m\da^2} {16M_J}= m\cdot \Bl(\f {\va} {8(J-j_0)} \Br)^2 \cdot \f 1 {16 M_J}\ge \f {c m\va^4}{R^p (\log \f R\va)^2 }.\]
	Since 
	\[L_j \ge N_{a(1+a)^j}(\CW_p, L_\infty),\]
this	implies 
	\begin{align*}
	S 
	&	\leq   \exp\Bl(-\f {cm\va^4} {R^p (\log \f R\va)^2}\Br) \sum_{j=j_0+1}^{J} \frac  1 {N_{a(1+a)^j}(\CW_p, L_\infty)}.\end{align*}

	On the other hand,  from  the  proof of Theorem 5.1 in  \cite{DT}, we know that 
	\begin{equation}\label{4-13}
	N_{t} (\CW_p, L_\infty) \ge  \f  {R} {4t},\   \   \   \forall   0<t<R.
	\end{equation} 
	It follows that
	\begin{align*}
S \leq   \exp\Bl(-\f {cm\va^4} {R^p (\log \f R\va)^2}\Br)\cdot 4c_\ast  R^{-1}\va    \sum_{j=j_0+1}^{J} (1+a)^j\leq \va  \exp\Bl(-\f {cm\va^4} {R^p (\log \f R\va)^2}\Br),
	\end{align*}
	which is as desired. This completes the proof.

\end{proof}

{\bf Acknowledgment.} The authors thank Tino Ullrich for bringing to their attention the paper \cite{HR}.  They would also like to express their gratitude to the referee  for carefully reading this paper and providing valuable suggestions and comments.

  \Addresses


\begin{thebibliography}{9999}
 
 \bibitem{BLM} J. Bourgain, J. Lindenstrauss and V. Milman, Approximation of zonoids by zonotopes,   {\it Acta Math.}, {\bf 162} (1989), 73--141.
 
 \bibitem{Bour} J. Bourgain, An improved estimate in the restricted isometry problem, In Geometric Aspects of Functional Analysis, volume 2116 of Lecture Notes in Mathematics, pages 65--70. Springer, 2014.

 \bibitem{CM} A. Cohen and G. Migliorati, Optimal weighted least-squares methods, {\it SMAI J. Computational Mathematics} {\bf 3} (2017), 181--203.
 
 
 

 \bibitem{DPTT} F. Dai, A. Prymak, V.N. Temlyakov, and  S.U. Tikhonov, Integral norm discretization and related problems,
  {\it Russian Math. Surveys} {\bf 74:4} (2019),   579--630.
 Translation from
{\it Uspekhi Mat. Nauk}  {\bf 74:4(448)}  (2019),	3--58; arXiv:1807.01353v1.
 
   \bibitem{DT} F. Dai and V. Temlyakov, Universal sampling discretization, arXiv:2107.11476v1 [math.FA] 23 Jul 2021.  Constr Approx (2023). https://doi.org/10.1007/s00365-023-09644-2
 
  
 
 \bibitem{DT3} F. Dai and V. Temlyakov, Universal discretization and sparse sampling recovery,
 arXiv:2301.05962v1 [math.NA] 14 Jan 2023.

 
 
 
  \bibitem{FR} S. Foucart and H. Rauhut, A Mathematical Introduction to Compressive Sensing,
  Birkh{\"a}user, 2013. 
  
  \bibitem{HR} I. Haviv and O. Regev, The restricted isometry property of subsampled Fourier matrices, In Geometric aspects of functional analysis, volume 2169 of Lecture Notes in Math., pages 163--179. Springer, Cham, 2017.
 
 \bibitem{JUV} T. Jahn, T. Ullrich, and F. Voigtlaender, Sampling numbers of smoothness classes via 
 $\ell^1$-minimization, arXiv:2212.00445v1 [math.NA] 1 Dec 2022. 
 
 
	
\bibitem{KKLT} B. Kashin, E. Kosov, I. Limonova, and V. Temlyakov, Sampling discretization and related problems, Journal of Complexity, {\bf 71} (2022), 101653; arXiv:2109.07567v1 [math.FA] 15 Sep 2021.


 
 

 
 
 
  \bibitem{LT} E. Livshitz and V. Temlyakov,  Sparse approximation and recovery by greedy algorithms, IEEE Transactions on Information Theory, {\bf 60} (2014), 3989--4000;
arXiv: 1303.3595v1 [math.NA] 14 Mar 2013.


\bibitem{S}
 W. M. Schmidt, Diophantine approximation. Lecture Notes in Mathematics {\bf  785}, Springer, 1980.




 
 


 




 \bibitem{T1} V.N. Temlyakov,  Weak greedy algorithms,
Adv. Comput. Math. \textbf{12}  (2000), 213--227.

 
 \bibitem{VTbook} V.N. Temlyakov, Greedy Approximation, Cambridge University
Press, 2011.

\bibitem{VT138} V.N. Temlyakov, An inequality for the entropy numbers and its application,
	{\it J. Approx. Theory }, {\bf 173} (2013), 110--121.

 \bibitem{VT150} V.N. Temlyakov, Constructive sparse trigonometric approximation and other problems for functions with mixed smoothness, arXiv: 1412.8647v1 [math.NA] 24 Dec 2014, 1--37; Matem. Sb., {\bf 206} (2015), 131--160. 
 
 
 \bibitem{VTbookMA} V. Temlyakov, {\em Multivariate Approximation}, Cambridge University Press, 2018.










  \end{thebibliography}
\end{document}